\documentclass[a4paper, 11pt, oneside]{scrartcl}
	\usepackage[english]{babel}
	\usepackage[utf8]{inputenc}
	\usepackage{paralist}
	\usepackage{mathtools}
	\usepackage{amssymb, amsthm}
	\usepackage{bbm}
	\usepackage{microtype}
	\usepackage{enumerate}
	\usepackage{mathrsfs}
	\usepackage{xcolor}
	\usepackage{graphicx}
	\usepackage{paralist}
	\usepackage[normalem]{ulem}
	\usepackage{tikz}
	\usetikzlibrary{matrix,arrows,positioning,shapes.geometric}
	\usepackage[square,numbers]{natbib}
	\usepackage{hyperref}
	\usepackage[capitalise]{cleveref}

	\theoremstyle{plain}
	\newtheorem{theorem}{Theorem}
	
	\newtheorem{prop}[theorem]{Proposition}
	\newtheorem{corollary}[theorem]{Corollary}
	 \newtheorem{conjecture}[theorem]{Conjecture} 
		
	\theoremstyle{definition}
	\newtheorem{defn}[theorem]{Definition}
	
	\theoremstyle{remark}
	\newtheorem{remark}[theorem]{Remark}

	\usepackage[left=3cm,right=3cm,top=3.5cm,bottom=4.2cm]{geometry}
	
	\DeclarePairedDelimiter{\sprod}{\langle}{\rangle}
	
	\DeclarePairedDelimiter{\fp}{\{}{\,|\kern-0.2em \}}
	\DeclarePairedDelimiter{\FK}{[}{\,|\kern-0.2em ]}
	\DeclarePairedDelimiter{\p}{\{}{\}}

	\newcommand{\E}{\mathrm {e}}
	\newcommand{\N}{\mathbb{N}}
	\newcommand{\K}{\mathbb{K}}

	\renewcommand{\phi}{\varphi}

	\newcommand{\im}{\operatorname{im}}
	\newcommand{\id}{\operatorname{id}}

	\newcommand{\del}{\partial}
	\renewcommand{\O}{\mathcal{O}}
	\newcommand{\factor}[2]{\left.\raisebox{.2em}{$#1$}\middle/\raisebox{-.2em}{$#2$}\right.}
	
	\newcommand{\ph}{[[t]]}

	\newcommand{\HC}{\mathrm{H\!C}}
	\newcommand{\HH}{\mathrm{H\!H}}
	\newcommand{\Hom}{\operatorname{Hom}}

	\newcommand{\Der}{\operatorname{Der}}

	\renewcommand{\id}{\operatorname{id}}

	\newcommand{\opp}{\mathrm{opp}}

	\newcommand{\1}{\mathbbm{1}}

	\newcommand{\delaa}{\del_{\alpha\alpha}}
	\newcommand{\delma}{\del_{\mu\alpha}}
	\newcommand{\delam}{\del_{\alpha\mu}}
	\newcommand{\delmm}{\del_{\mu\mu}}
	\newcommand{\delmb}{\del_{\bar{\mu}}}
	
	\newcommand{\aHH}{\mathrm{\tilde{H}\!H}}
	\newcommand{\aHC}{\mathrm{\tilde{H}\!C}}
	\newcommand{\aHZ}{\mathrm{\tilde{H}\!Z}}
	\newcommand{\aHB}{\mathrm{\tilde{H}\!B}}

	\newcommand{\vera}[1]{%
	\begin{tikzpicture}[scale=0.3,point/.style={draw,shape=circle,fill=blue,minimum size=2,inner sep=0}, circ/.style={draw,shape=circle,minimum size=5,inner sep=0}]
	 \node [circ] (a0) at ( #1/2 + 0.5,1) {};
	 \draw (a0) -- +(0,0.7) ;
	 \foreach  \i in{1,...,#1}
	 {
	\draw (\i,0) -- (a0);
	}
	\end{tikzpicture}%
	}
	
	\newcommand{\verm}[1]{%
	\begin{tikzpicture}[scale=0.3,point/.style={draw,shape=circle,fill=blue,minimum size=2,inner sep=0},  circ/.style={draw,shape=circle,minimum size=5,inner sep=0}]
	 \node [point] (a0) at ( #1/2 + 0.5,1) {};
	 \draw (a0) -- +(0,0.7) ;
	 \foreach  \i in{1,...,#1}
	 {
	\draw (\i,0) -- (a0);
	}
	\end{tikzpicture}%
	}
	
	\newcommand{\vermk}{%
	\begin{tikzpicture}[scale=0.3,point/.style={draw,shape=circle,fill=blue,minimum size=2,inner sep=0},  circ/.style={draw,shape=circle,minimum size=5,inner sep=0}]
	 \node [point] (a0) at ( 2,1.3) {};
	 \draw (a0) -- +(0,0.7) ;
	 \draw (1,0) -- (a0);
	 \draw (3,0) -- (a0);
	 \draw[dotted] (1,0) -- node[above,scale=.7] {$k$} (3,0);
	\end{tikzpicture}%
	}
	
	\newcommand{\verak}{%
	\begin{tikzpicture}[scale=0.3,point/.style={draw,shape=circle,fill=blue,minimum size=2,inner sep=0},  circ/.style={draw,shape=circle,minimum size=5,inner sep=0}]
	 \node [circ] (a0) at ( 2,1.3) {};
	 \draw (a0) -- +(0,0.7) ;
	 \draw (1,0) -- (a0);
	 \draw (3,0) -- (a0);
	 \draw[dotted] (1,0) -- node[above,scale=.7] {$k$} (3,0);
	\end{tikzpicture}%
	}

	\allowdisplaybreaks

	\title{$\alpha$-type Hochschild cohomology of Hom-associative algebras and Hom-bialgebras}
	\author{Benedikt Hurle and Abdenacer Makhlouf}

\begin{document}
\maketitle
	
\begin{abstract}
	In this paper we define a new cohomology for multiplicative Hom-associative algebras, which generalize Hochschild cohomology and fits with deformations of Hom-associative algebras including the structure map  $\alpha$. 
	It is a generalization of the known Hochschild-type cohomology for Hom-associative algebras  \citep{homcoho}. 
	Moreover, we provide various observations and similarly a new cohomology of Hom-bialgebras extending the Gerstenhaber-Schack cohomology for Hom-bialgebras given in \citep{MR3640817}.
\end{abstract}

\tableofcontents

\section*{Introduction}The first instances of Hom-type algebras appeared when $q$-deforming algebras of vector fields like Witt and Virasoro algebras, that is replacing usual derivation by Jackson derivation, see \citep{silvestrov06}. 	
Hom-associative algebras, generalizing associative algebras,  were introduced  in \citep{homalg}. The main feature of Hom-type algebras is that classical identities are twisted by a homomorphism, usually denoted by $\alpha$. Many concepts and results were extended to Hom-type algebras. An attempt of building a cohomology for Hom-associative algebras together with a deformation theory were considered in \citep{makhlouf_ParamFormDefHomAssLie}.
In  \citep{homcoho},  a Hochschild-type cohomology complex was constructed for multiplicative Hom-associative algebras. This cohomology  fitted  with a 1-parameter formal deformation theory where the multiplication is deformed and the structure map kept fixed. 

In this paper  a new type of Hochschild cohomology is provided  for Hom-associative algebras and Hom-bialgebras, taking into account the structure map $\alpha$. It deals with  pairs in the complex and  allows to study deformations, where the product but also the structure map $\alpha$ are  deformed. This is why we call this cohomology $\alpha$-type Hochschild cohomology.
This approach can also be extended to other types of Hom-algebras, we will consider the case of Hom-Lie algebras in a forthcoming paper.
	
The paper is organized  as follows.
First we recall some basics,  the definitions of Hom-associative algebras and Hom-coalgebras, as well as their corresponding module and comodule structures.  In Section 3, we define the $\alpha$-type Hochschild cohomology complex and provide some observations. We compute this cohomology for associative algebras, viewed as  Hom-associative algebra ($\alpha=\id$), and show how it relates to the cohomology given in \citep{homcoho}.
Moreover, we study  its behaviour under Yau twist,  give a way of computing it if $\alpha$ is invertible and 
 discuss its corresponding $L_\infty$-structure. 	
In \Cref{sc:deformation}, we show that  deformations of a Hom-associative algebras, including deformation of the structure map,  are governed by this new cohomology. In \Cref{sc:examples},  we compute the cohomology for some examples.
Finally, in  \Cref{sc:Hombialgeras}, we  extend this cohomology to the Hom-bialgebras case, by providing a cohomology complex. A complete study for Hom-bialgebras will be done in a different paper.

\section{Preliminaries}
	
Let $\K$ be a field of characteristic 0. One could also consider a commutative ring here or a field of another characteristic, but for the sake of simplicity we stick to  this case. All vector spaces and algebras will be considered over this field.
We begin with the definition of Hom-associative algebras, which originates along with the definitions of other types of Hom-algebras in \citep{homalg}.	
	
\begin{defn}[Hom-(associative) algebra]
	A Hom-associative algebra $(A,\mu,\alpha)$, is a vector space  $A$, with a linear map $\alpha: A \to A$, called the structure map, and a linear map $\mu:A \otimes A \to A$, called multiplication, such that 
	\begin{equation}
		\mu\circ ( \id \otimes \alpha) =  \mu \circ ( \alpha \otimes \id).
	\end{equation}
	It is called multiplicative if $\alpha$ is an algebra morphism, i.e.
	\begin{equation}
		\mu \circ(\alpha \otimes \alpha) = \alpha\circ \mu.
	\end{equation}
	Note that this property is not required by all authors.
	In this paper however, we will consider only multiplicative Hom-associative algebras and simply call them Hom-associative algebras.
	A unit is an element $\1  \in A$ such that $ \mu(\1 \otimes \id)= \alpha = \mu(\id \otimes \1)$  and $\alpha(\1) = \1$.
\end{defn}
Note that the unit can also be considered as a map $\K \to A$ determined by $1_\K \mapsto \1_A$, what we will do in the following if it is convenient.
	
Usually we write $x y $ or $x \cdot y$ for $\mu(x \otimes y)$, so the Hom-associativity 
becomes $\alpha(x)(yz) = (xy)\alpha(z)$ for all $x,y,z \in A$, and the condition for the unit $\1 \cdot x = \alpha(x) = x \cdot \1$.

A morphism of Hom-associative algebras $(A,\mu_A,\alpha_A)$ and $(B,\mu_B,\alpha_B)$ is a linear map $ \phi : A \to B$, such that 
$\phi\circ  \alpha_A = \alpha_B\circ  \phi$ and $ \mu_B\circ  (\phi \otimes \phi) = \phi \circ  \mu_A$.
	
The tensor product of two Hom-associative algebras is again a Hom-associative algebra, with multiplication given by 
$(a \otimes b) ( a' \otimes b') = (aa') \otimes (bb')$ and structure map $\alpha \otimes \beta$.

\begin{remark}[Yau twist]
	Let $(A,\mu,\alpha)$ be a Hom-associative algebra and $\gamma:A\rightarrow A$ be an algebra morphism, then $(A,\gamma\circ \mu, \gamma \circ \alpha)$ is again a Hom-associative algebra, which we denote by $A_\gamma$.
\end{remark}
	
If a Hom-associative algebra is of the form $A_\gamma$ for an associative algebra $A$ and an endomorphism $\gamma$ of it, we call it of associative-type.

\begin{prop}\label{th:assass}
	Let $(A,\mu,\alpha)$ be a Hom-associative algebra such that $\alpha$ is invertible, then $A_{\alpha^{-1}}$ is an ordinary associative algebra. So $A$ is of associative-type.
\end{prop}
\begin{proof}
	Follows directly from the previous remark.
\end{proof}

\begin{defn}[Derivation, $\alpha$-derivation]
	Given a Hom-associative algebra  $A$,  a derivation is a linear map $\phi: A \to A$ such that 
	\begin{equation}
		\phi(xy) = \phi(x)y + x\phi(y)
	\end{equation}
	and an $\alpha$-derivation is a linear map $\phi: A \to A$, such that
	\begin{equation}
		\phi(xy) = \phi(x) \alpha(y) + \alpha(x)\phi(y).
	\end{equation}
\end{defn}
	
For a derivation or an $\alpha$-derivation $\phi$ we have  $2 \alpha(\psi(\1)) = \psi(\1)$, from which one can only conclude that $\psi(\1)$ is an eigenvector of $\alpha$  with eigenvalue $\frac{1}{2}$. However if $\phi(\1) =0$, which is usually the case, we get $\phi(\alpha(x))=\phi(\1 x) = \phi(\1)\alpha(x) + \alpha(\1) \phi(x) = \alpha ( \phi(x))$, so $\alpha$ and $\phi$ commute. 
We want to give an example of $\alpha$-derivation, where $\phi(\1) \neq 0$. 
Consider the polynomial algebra $\K[x]$ and the morphism $\alpha$ of it, determined by $\alpha(x) = \frac{1}{2} x$. On the Hom-associative algebra $(\K[x], \cdot, \alpha)$ obtained by Yau twist with $\alpha$, we define an $\alpha$-derivation by $\phi(p) = x \alpha(p)$ for $p \in\K[x]$. In fact $\phi(p) \cdot \alpha(q) = \alpha( x\alpha(p) \alpha(q)) = \frac{1}{2}x \alpha^2(pq) = \alpha(p) \cdot \phi(q)$ and $\phi(p\cdot q) = x \alpha^2(pq)$ for $p,q \in \K[x]$.\\

For later use, we also define conjugate $\alpha$-derivation.
\begin{defn}[conjugate $\alpha$-derivation]
	Let $A$ be a Hom-associative algebra, such that $\alpha$ is invertible. Then a linear map $\phi: A\to A$ is called a conjugate $\alpha$-derivation if it satisfies 
	\begin{equation}
		\phi(x y) = \alpha(x) \alpha^{-1}\phi \alpha (y) + \alpha^{-1}\phi \alpha(x) \alpha(y).
	\end{equation}
\end{defn}
	
Obviously if $\alpha$ and $\phi$ commute a conjugate $\alpha$-derivation is the same as an $\alpha$-derivation. For a conjugate $\alpha$-derivation $\phi$, one has $\phi(\1) =0$. \\

Dually to the concept of Hom-associative algebras, there exists also  the concept of Hom-coassociative coalgebra. 
	
\begin{defn}[Hom-(coassociative) coalgebra]
	A Hom-(coassociative) coalgebra $(C,\Delta,\beta)$ is a vector space $C$ with a linear map $\beta: C \to C$, called the structure map, and  a comultiplication $\Delta: C \to C \otimes C$, which satisfies 
	\begin{equation}
		(\Delta \otimes \beta)\circ  \Delta = (\beta \otimes \Delta) \circ \Delta.
	\end{equation}
	It is called comultiplicative if $\Delta\circ  (\beta \otimes \beta) = \beta \circ \Delta$. In this paper we only consider comultiplicative Hom-coalgebras and simply call them Hom-coalgebras.
	A counit is a map $\epsilon:C \to \K$, such that $(\id \otimes \epsilon)\circ \Delta = \beta = (\epsilon \otimes \id) \circ \Delta$ and $ \epsilon\circ  \beta = \epsilon$.
\end{defn}
We denote the structure map here by $\beta$ and not by $\alpha$, since when dealing with bialgebras we want to distinguish them.\\

A morphism of Hom-coalgebras $(C,\Delta_C,\beta_C)$ and $(D,\Delta_D,\beta_D)$ is a linear map $ \phi : C \to D$, such that 
$\phi\circ \beta_C = \beta_D\circ  \phi$ and $\Delta_D\circ  \phi = (\phi \otimes \phi)\circ \Delta_C$.	
	
\begin{remark}
	Given a Hom-coalgebra $(C,\Delta,\beta)$ and an endomorphism $\gamma: C\rightarrow C$, the triple $(C,\Delta\gamma,\beta\gamma)$ is a again a Hom-coalgebra. We call it the Yau twist of $C$ by $\gamma$. If a Hom-coalgebra is the Yau twist of a coassociative coalgebra, we call it of associative-type.
\end{remark}

\begin{defn}[Hom-(associative) bialgebra]
	A Hom-(associative) bialgebra is tuple $(A,\mu,\Delta,\1, \epsilon, \alpha, \beta)$ such that 
	$(A,\mu,\1, \alpha)$ is a   Hom-associative algebra, $(A,\Delta, \epsilon, \beta)$ is a  Hom-coalgebra  and they are compatible in the sense that $\Delta$ is a Hom-algebra morphism and $\mu$ is a Hom-coalgebra morphism. Also the unit is assumed to be a Hom-coalgebra morphism and the counit a Hom-algebra morphism. We  also require that $\alpha$ and $\beta$ commute. Written as identities, this means
	\begin{align*}
		\alpha \circ \beta & = \beta \circ \alpha,               &   
		\Delta(xy) &= \Delta(x) \Delta(y), \\
		\Delta\circ   \alpha     & =   (\alpha \otimes \alpha) \circ \Delta, &   
		\beta\circ  \mu &= \mu\circ  ( \beta \otimes \beta), \\
		\Delta(\1)          & = \1 \otimes \1,                      &   
		\epsilon(xy) &= \epsilon(x) \epsilon(y)\\
		\beta(\1) &= \1, &  \epsilon \circ \alpha &=\epsilon, \\
		\epsilon(\1) &= 1 .
	\end{align*}
\end{defn}

Often only the case where $\alpha = \beta $ or $\beta = \alpha^{-1}$ is considered, but here we distinguish between the two structure maps since it will make the construction of the cohomology clearer.

By a morphism between two Hom-bialgebras $A$ and $B$, we mean a linear map $\phi: A \to B$, which is a morphism of the corresponding Hom-algebra and Hom-coalgebra structures.
	
\begin{remark}[Yau twist]
	Let $\gamma$ be a morphism of a Hom-bialgebra  $(A,\mu,\Delta,\1, \epsilon, \alpha, \beta)$,  then also $(A,\gamma \mu,\Delta,\1, \epsilon, \gamma \alpha, \beta)$, $(A,\mu,\Delta\gamma,\1, \epsilon, \alpha, \gamma\beta)$ and  $(A,\gamma\mu,\Delta\gamma,\1, \epsilon, \gamma\alpha, \gamma\beta)$ are Hom-bialgebras. 
\end{remark}

A Hom-bialgebra $H$ is called a Hom-Hopf algebra if it possesses an antipode, i.e. an anti-homomorphism  $S:H \to H$ which satisfies, $ \mu (S \otimes \id)  \Delta =   \alpha \1 \epsilon  \beta = \mu (\id \otimes S)  \Delta$. 
There is also the concept of  Hom-Hopf algebras, where the condition for the antipode is relaxed, see e.g. \citep{UEHomLie}, but we will not need this in the sequel.

\section{Modules, Comodules and Bimodules}
	
In this section we  give the basic definitions and properties of modules and bimodules for Hom-associative algebras, and dually comodules for Hom-coalgebras.

\begin{defn}[Hom-algebra module]
	Let $(A,\mu,\alpha)$ be a Hom-associative algebra and $(M,\beta)$ be a Hom-module, i.e. a vector space $M$ with a linear map $\beta:M \to M$. Further let $\rho:A \otimes M \to M, (a \otimes m)\mapsto a \cdot m$ be a linear map, then $(M,\rho)$ is called a (left) $A$-module  if 
	\begin{align}
		(a b) \cdot \beta(m) & = \alpha( a ) ( b \cdot m). 
	\end{align}
	Since we are considering multiplicative Hom-associative algebras, we also assume 
	\begin{equation}
		\beta(a \cdot m) = \alpha(a) \cdot \beta(m).
	\end{equation}
	This property is called multiplicativity.
\end{defn}
Similarly one can define a right $A$-module.
Obviously $A$ is an $A$-module, where the action is giving by the multiplication in $A$.

\begin{defn}[Hom-algebra bimodule]
	Let $A$  be a Hom-associative algebra, then an $A$-bimodule is a Hom-module $(M,\beta)$, with two maps $\rho: A \otimes M \to M, a \otimes m \mapsto a \cdot m$ and $\lambda: M \otimes A \to M, a \otimes m \mapsto m \cdot a$, such that $\rho$ is a left and $\lambda$ a right module structure and 
	\begin{equation}
		\alpha(a) \cdot( m \cdot b) = ( a \cdot m) \cdot \alpha(b).
	\end{equation}
\end{defn}

It is easy to see that if $V$ is an $A$-bimodule then $A  \oplus V$ is a Hom-associative algebra.

Given a Hom-coassociative coalgebra one can regard its comodules, dual to the case of Hom-associative algebras and modules. 
	
\begin{defn}[Hom-coalgebra comodule]
	Let $(C,\Delta,\beta)$ be a Hom-coalgebra and $(M, \gamma)$ be a  Hom-module, with a map $\rho: M \to A \otimes M$, then $M$ is called a (left) $C$-comodule  if 
	\begin{align}
		( \beta \otimes \rho) \rho = (\Delta \otimes \beta) \rho. 
	\end{align}
	We also assume that it is comultiplicative, i.e. 
	\begin{equation}
		\rho \gamma = (\beta \otimes \gamma) \rho.
	\end{equation}
\end{defn}
Similarly one can define a right $C$-comodule.
\begin{defn}[Hom-coalgebra bicomodule]
	Let $(C,\Delta,\beta)$  be a Hom-coalgebra, then a $C$-bicomodule is a Hom-module $(M,\gamma)$, with two maps $\rho:  M \to A \otimes M$ and $\lambda: M \otimes A \to M \otimes A, a \otimes m \mapsto m \cdot a$, such that $\rho$ is a left and $\lambda$ a right comodule structure and 
	\begin{equation}
		( \beta \otimes  \rho) \lambda = (\lambda \otimes \beta) \rho.
	\end{equation}
\end{defn}

Let $A$ be a Hom-bialgebra then the tensor product of two $A$-modules $M$ and $N$ is again an $A$-module, where the action is given by $(\rho_M \otimes \rho_N) \tau_{23}  (\Delta \otimes \id \otimes \id)$ or 
\begin{equation}
	a \cdot (m \otimes n) = a_{(1)} \cdot m \otimes a_{(2)} \cdot n.   
\end{equation}
	
Here we made use of the Sweedler notation, i.e. $\Delta a = a_{(1)} \otimes a_{(2)}$, where on the right hand side there is an implicit sum.

Note that given three or more $A$-modules the module structure on the tensor product depends in general on the bracketing of the modules, i.e. $(M_1 \otimes M_2) \otimes M_3$ is not isomorphic to  $M_1 \otimes (M_2 \otimes M_3)$, since the comultiplication is not coassociative. 
But we can define an action on tensor products with more factors independent of the bracketing  by including $\beta$ in the definition. We will briefly recall this here for details see \citep{MR3640817}.  
For this we define $\Delta^2_\beta = \Delta$ and 
\begin{equation}
	\Delta^{n+1}_\beta  = (\Delta \otimes \beta^{\otimes (n-1)} ) \Delta^n_\beta.
\end{equation}
This satisfies 
\begin{equation}
	\Delta^{n+m}_\beta  = (\Delta^{n}_\beta  \beta^{m-1} \otimes \Delta^{m}_\beta \beta^{n-1}) \Delta. \label{eq:mu1}
\end{equation}
Let $(M_i,\rho_i)$ be $A$-modules then the action on $M_1 \otimes \dots \otimes M_n$ is defined by
\begin{equation}
	(\rho_1 \otimes\dots \otimes \rho_n ) \tau_{(2,n)}( \Delta^n_\beta \otimes \id^n).
\end{equation}
Here $ \tau_{(2,n)}$ denotes the permutation $(1,\dots ,2n) \mapsto (1,n+1,2,n+2,\dots)$.
	
We also have 
\begin{equation}
	x \cdot ( v_1 \otimes \dots \otimes v_{m+n}) = \left(  \beta^{m-1}(x) \cdot (v_1\otimes \dots \otimes v_n) \right) \otimes   \left(  \beta^{n-1}(x) \cdot (v_{n+1}\otimes \dots \otimes v_{n+m}) \right) \label{eq:mod1}
\end{equation}
If all  the $M_i$ are bimodules the tensor product is again a bimodule.
Similarly the tensor product of comodules is again in a comodule.

\section{$\alpha$-type Hochschild cohomology of Hom-associative algebras}\label{Section3}
	
In this section we define a  cohomology for Hom-associative algebras, which takes also into account the structure map $\alpha$.
We will see that it is a generalization of the cohomology which  is usually considered, see \citep{homcoho}.  
The complex is different from the one used  for the Hochschild cohomology, in the sense that the cochains are given by pairs. In the case of associative algebras, i.e. $\alpha = \id$, it is completely determined by the Hochschild cohomology, but in the general case it contains more information. It extends the cohomology given in \citep{makhlouf_ParamFormDefHomAssLie}. 
With this cohomology it is also possible to consider deformations of Hom-associative algebras, where the multiplication and the structure map are deformed. 
	
We start by giving the complex for the cohomology of a Hom-associative algebra $A$ with values in itself. 
The cochains are given by 
\begin{eqnarray}
	\aHC^n(A) = \aHC^n_\mu(A) \oplus \aHC^n_\alpha(A) = \Hom(A^{\otimes n}, A) \oplus \Hom(A^{\otimes (n-1)},A) \text{ for } n \geq 2,
\end{eqnarray}
$\aHC^1(A) = \aHC^1_\mu(A) \oplus \aHC^1_\alpha(A) = \Hom(A, A) \oplus \{ 0 \}$ and $\aHC^n(A) =\{0\}$ for $n \leq 0$.
In general we will denote a $n$-cochain by the pair $(\phi,\psi)$, where  $\phi \in \aHC_\mu^n$  and $\psi \in \aHC^n_\alpha$.
We have $\alpha \in \aHC^2_\alpha$ and $\mu \in \aHC^2_\mu$, which motivates the names for the two summand. 
	
\begin{remark}
	Note that $\aHC^0$ is the zero space and not as one might expect $\Hom(\K,A) \cong A$. This is so, because otherwise the differential would involve $\alpha^{-1}$, what we do not want, since we consider $\alpha$ to be not necessarily invertible. The same is true for $\aHC^1_\alpha$.  If $\alpha$ is invertible one can add these components and the corresponding differentials. In this case also \Cref{th:cohoass} simplifies. But we will not consider this further.
\end{remark}
	
We define four maps, with domain and range given in the following diagram :
\begin{center}
	\begin{tikzpicture}
		\matrix (m) [matrix of math nodes,row sep=3em,column sep=5em,minimum width=2em] {
			\aHC^n_\mu & \aHC^{n+1}_\mu  \\
			\aHC^n_\alpha & \aHC^{n+1}_\alpha \\
		};  
		\path[->,auto] (m-2-1) edge node[swap]{$\delaa$}                  (m-2-2);
		\path[->,above] (m-2-1) edge node[below] {$\delam$}                  (m-1-2);
		\path[->,below] (m-1-1) edge node[above] {$\delma$}                  (m-2-2);
		\path[->,auto] (m-1-1) edge node {$\delmm$}                  (m-1-2);
		\path   (m-2-1) edge[draw=none]   node [sloped] {$\oplus$} (m-1-1);
		\path   (m-2-2) edge[draw=none]   node [sloped] {$\oplus$} (m-1-2);
	\end{tikzpicture}
\end{center}
	
First the classical differential for Hom-algebras $\delmm : \aHC^{n}_\mu \to \aHC^{n+1}_\mu$
\begin{align}
	\begin{split}
	\delmm \phi (x_1, \ldots, x_{n+1}) & =\alpha^{n-1}(x_1) \phi(x_2, \ldots, x_{n+1})                                            \\
	                                   & + \sum_{i=1}^{n} (-1)^{i} \phi(\alpha(x_1), \dots, x_i x_{i+1}, \dots , \alpha(x_{n+1})) \\
	                                   & + (-1)^{n+1} \phi(x_1, \ldots, x_{n-1}) \alpha^{n-1}(x_{n+1}).                           
	\end{split}
\end{align}
The map $\delaa : \aHC^{n}_\alpha \to \aHC^{n+1}_\alpha$ is  also the classical differential for Hom-associative algebras, but the bimodules structure is modified by $\alpha$, this means
\begin{align}
	\begin{split}
	\delaa \psi (x_1,\ldots, x_n) & = \alpha^{n-1}(x_1) \psi(x_2, \ldots, x_n)                                             \\
	                              & + \sum_{i=1}^{n-1} (-1)^{i} \psi(\alpha(x_1), \dots, x_i x_{i+1}, \dots , \alpha(x_n)) \\
	                              & + (-1)^{n} \psi(x_1, \ldots, x_{n-1}) \alpha^{n-1}(x_n),                               
	\end{split}
\end{align}
the map $\delma : \aHC^{n}_\mu \to \aHC^{n+1}_\alpha$ is the commutator of $\alpha$ and $\phi$ defined by 
\begin{equation}
	\delma \phi (x_1,\ldots, x_n)  = \alpha( \phi(x_1,\ldots, x_n) ) - \phi(\alpha(x_1),\ldots,\alpha( x_n)) 
\end{equation}
and  finally $\delam : \aHC^{n}_\alpha \to \aHC^{n+1}_\mu$ is defined by
\begin{equation}
	\delam \psi (x_1,\ldots, x_{n+1}) = \alpha^{n-2}(x_1 x_2)  \psi(x_3, \ldots, x_{n+1})  -  \psi(x_1, \ldots, x_{n-1}) \alpha^{n-2}(x_{n} x_{n+1}),
\end{equation}
for $x_1, \ldots, x_{n+1}$ in $A$.\\
	
With this we set
\begin{align}
	\del (\phi + \psi) & = (\delmm + \delma) \phi - (\delam + \delaa )\psi \nonumber \\
	                   & = ( \delmm \phi - \delam \psi, \delma \phi - \delaa \psi).   
\end{align}
Note that we are considering a difference  in this formula. We have chosen this convention, because so the single differentials look more natural.  \\
	
Therefore, we get the following theorem : 
	
\begin{theorem}\label{th:delalg}
	The complex $\aHC^\bullet(A)$ is a chain complex with differential  $\del(\phi,\psi) = (\delmm + \delma) \phi - (\delam + \delaa )\psi$ defined as above.
\end{theorem}
\begin{proof}
	We only need to show $\del^2 =0$.
	This is a lengthy but straightforward calculation. We will give it here to some extend.

	\begin{align}
		\delmm \delmm \phi & (x_1, \dots, x_{n+2})  =                                                                                                                                                 
		\alpha^n(x_1) (\delmm \phi) (x_2,\dots, x_{n+2})  \nonumber\\
		                   & + \sum_{i=1}^{n+1}(-1)^i (\delmm \phi)(\alpha(x_1), \dots, x_i x_{i+1}, \dots,\alpha(x_{n+2}))  \nonumber                                                                \\
		                   & + (-1)^n (\delmm \phi) (x_1, \dots, x_{n+1}) \alpha^n(x_{n+2})  \nonumber                                                                                                \\
		                   & = \alpha^n(x_1) ( \alpha^{n-1}(x_2) \phi(x_3, \dots,x_{n+2}))   \nonumber                                                                                                \\
		                   & - \sum_{i=2}^{n+1} (-1)^i \alpha^n(x_1) \phi(\alpha(x_2),\dots,x_i x_{i+1},\dots, \alpha(x_{n+2}))  \label{eq:dat41}                                                     \\
		                   & - (-1)^n \alpha^n(x_1)(\phi(x_2,\dots,x_{n+1}) \alpha^{n-1}(x_{n+2}))  \label{eq:dat61}                                                                                  \\
		                   & + \alpha^{n-1}(x_1 x_2) \phi(\alpha(x_3),\dots,\alpha(x_{n+2}))  \nonumber                                                                                               \\
		                   & + \sum_{i=2}^{n+1} (-1)^i \alpha^n(x_1) \phi(\alpha(x_2),\dots,x_i x_{i+1},\dots, \alpha(x_{n+2}))  \label{eq:dat42}                                                     \\
		                   & + \sum_{i=1}^{n+1} \sum_{j=1}^{i-2} (-1)^{i+j} \phi(\alpha^2(x_1),\dots, \alpha(x_j x_{j+1}),\dots, \alpha(x_i x_{i+1}), \dots, \alpha^2(x_{n+2}))  \label{eq:dat1}      \\
		                   & - \sum_{i=2}^{n+1} \phi(\alpha^2(x_1),\dots, \alpha(x_{i-1}) (x_i x_{i+1}), \dots, \alpha^2(x_{n+2}))  \label{eq:dat21}                                                  \\
		                   & + \sum_{i=1}^{n}  \phi(\alpha^2(x_1),\dots, (x_{i-1} x_i) \alpha( x_{i+1}), \dots, \alpha^2(x_{n+2}))  \label{eq:dat22}                                                  \\
		                   & + \sum_{i=1}^{n+1} \sum_{j=i+2}^{n+1} (-1)^{i+j-1} \phi(\alpha^2(x_1),\dots, \alpha(x_i x_{i+1}),\dots, \alpha(x_j x_{j+1}), \dots, \alpha^2(x_{n+2}))  \label{eq:dat11} \\
		                   & + \sum_{i=1}^{n} (-1)^{i+n+1} \phi(\alpha(x_1),\dots,x_i x_{i+1},\dots, \alpha(x_{n+1}))\alpha^n(x_{n+2}) \label{eq:dat51}                                               \\
		                   & +  \phi(\alpha(x_1),\dots, \alpha(x_{n}))\alpha^{n-1}(x_{n+1}x_{n+2})  \nonumber                                                                                         \\
		                   & + (-1)^n (\alpha^n(x_1)\phi(x_2,\dots,x_{n+1})) \alpha^{n}(x_{n+2})  \label{eq:dat62}                                                                                    \\
		                   & + \sum_{i=1}^{n} (-1)^{i+n} \phi(\alpha(x_1),\dots,x_i x_{i+1},\dots, \alpha(x_{n+1})) \alpha^n(x_{n+2})  \label{eq:dat52}                                               \\
		                   & - (\phi(x_1, \dots,x_{n}) \alpha^{n-1}(x_{n+1}))  \alpha^{n}(x_{n+2}). \nonumber                                                                                          
	\end{align}
	The terms \ref{eq:dat21} cancels with \ref{eq:dat22}, \ref{eq:dat41}  with \ref{eq:dat42}, \ref{eq:dat51}  with \ref{eq:dat52} and  \ref{eq:dat61}  with \ref{eq:dat62}. Exchanging the indices in \ref{eq:dat11}, one sees that it cancels with \ref{eq:dat1}. The remaining four terms can be easily arranged to 
	\begin{align}
		\delam \delma \phi(x_1, \dots, x_{n+2}) & =	\alpha^{n-1}(x_1 x_2)(\delma \phi)(x_3,\dots,x_{n+2})                               \\
		                                        & -  (\delma \phi) (x_1, \dots, x_n) \alpha^{n-1}(x_{n+1}x_{n+2})  \nonumber            \\
		                                        & = \alpha^{n-1}(x_1 x_2) \alpha(\phi(x_3,\dots,x_{n+2})) \label{eq:dat14}              \\
		                                        & - \alpha^{n-1}(x_1 x_2) \phi(\alpha(x_3),\dots,\alpha(x_{n+2}))\label{eq:dat12}       \\
		                                        & -  \alpha(\phi(x_1,\dots,x_{n}))\alpha^{n-1}(x_{n+1}x_{n+2}) \label{eq:dat18}         \\
		                                        & + \phi(\alpha(x_1),\dots,\alpha(x_{n})) \alpha^{n-1}(x_{n+1}x_{n+2}). \label{eq:dat16} 
	\end{align}
												
	Next we check $\delmm \delam = \delam \delaa$ explicitly. On one side we get 
	\begin{align}
		\delmm \delam \psi & (x_1, \ldots, x_{n+2})  = \alpha^n(x_1)(\delam \psi)(x_2,\ldots,x_{n+2})  \nonumber                                              \\
		                   & + \sum_{i=1}^{n+1} (-1)^i (\delam \psi)( \alpha(x_1), \dots, x_i x_{i+1},\dots ,\alpha(x_{n+2})  \nonumber                       \\
		                   & + (-1)^{n+2} (\delam \psi)(x_1, \dots,x_{n+1}) \alpha(x_{n+2}) \nonumber                                                         \\
		                   & = \alpha^n(x_1) (\alpha(^{n-2}(x_2x_3) \psi(x_4,\dots,x_{n+2}) \label{eq:das13}                                                  \\
		                   & -  \alpha^n(x_1) ( \psi(x_2,\dots,x_n) \alpha^{n-2}(x_{n+1}x^{n+2}) \label{eq:das15}                                             \\
		                   & - (\alpha^{n-2}(x_1x_2)\alpha^{n-1}(x_3)) \psi(\alpha(x_4), \dots, \alpha(x_{n+2}) \label{eq:das1}                               \\
		                   & + (\alpha^{n-1}(x_1)\alpha^{n-2}(x_2x_3)) \psi(\alpha(x_4), \dots, \alpha(x_{n+2}) \label{eq:das2}                               \\
		                   & +\sum_{i=3}^{n+1} (-1)^i \alpha^{n-1}(x_1x_2) \psi(\alpha(x_3), \dots,x_ix_{i+1},\dots, \alpha(x_{n+2})) \label{eq:das3}         \\
		                   & -  \sum_{i=1}^{n-1} (-1)^i \psi(\alpha(x_1), \dots,x_ix_{i+1},\dots, \alpha(x_{n}))\alpha^{n-1}(x_{n+1}x_{n+2})  \label{eq:das5} \\
		                   & - (-1)^n\psi(\alpha(x_1),\dots,\alpha(x_n))(\alpha^{n-2}(x_nx_{n+1})\alpha^{n-1}(x_{n+2})) \label{eq:das7}                       \\
		                   & + (-1)^n \psi(\alpha(x_1),\dots,\alpha(x_n))(\alpha^{n-1}(x_n)\alpha^{n-2}(x_{n+1}x_{n+2})) \label{eq:das8}                      \\
		                   & + (-1)^{n+2} (\alpha^{n-2}(x_1 x_2) \psi(x_3,\dots,x_{n+1})) \alpha^n(x_{n+2})\label{eq:das11}                                   \\
		                   & - (-1)^{n+2} (\psi(x_1,\dots,x_{n-1}) \alpha(x_nx_{n+1}))\alpha^n(x_{n+2}) \label{eq:das9}.                                       
	\end{align} 
	On the other side  one has 
	\begin{align}
		\delam \delaa  \psi & (x_1, \ldots, x_{n+2})=  \alpha^{n-1} (x_1x_2) (\delaa \psi)(x_3,\dots,x_{n+2}) \nonumber                              \\
		                    & - (\delaa \psi)(x_1, \dots, x_n) \alpha^{n-1}(x_{n+1}x_{n+2}) \nonumber                                                \\
		                    & =\alpha^{n-1}(x_1x_2) (\alpha^{n-1}(x_3) \psi(x_4, \dots, x_{n+2}))  \label{eq:das14}                                  \\
		                    & +\sum _{i=3}^{n+1} (-1)^i \alpha^{n-1}(x_1 x_2) \psi(\alpha(x_3,\dots,x_ix_{i+1},\dots,\alpha(x_{n+2}))\label{eq:das4} \\
		                    & + (-1)^{n+2} \alpha^{n-1}(x_1 x_2) (\psi(x_3,\dots, x_{n+1})\alpha^{n-1}(x_{n+2})) \label{eq:das12}                    \\
		                    & -  (\alpha^{n-1}(x_1) \psi(x_2,\dots,x_n)) \alpha(x_{n+1}x_{n+2})) \label{eq:das16}                                    \\
		                    & - \sum_{i=1}^{n-1} \psi(\alpha(x_1),\dots,x_i x_{i+1},\dots,\alpha(x_n)) \alpha^{n-1}(x_{n+1} x_{n+2}) \label{eq:das6} \\
		                    & - (-1)^n (\psi(x_1,\dots,x_n) \alpha^{n-1}(x_n))\alpha^{n-1}(x_{n+1}x_{n+2}). \label{eq:das10}                          
	\end{align}
	Now \eqref{eq:das1} cancels with \eqref{eq:das2} and \eqref{eq:das7} with \eqref{eq:das8}.
	Using the Hom-associativity  it is also  easy to see that \eqref{eq:das3} equals \eqref{eq:das4},
	\eqref{eq:das5} equals \eqref{eq:das6}, \eqref{eq:das9} equals \eqref{eq:das10},
	\eqref{eq:das11} equals \eqref{eq:das12}, \eqref{eq:das13} equals \eqref{eq:das14} and  \eqref{eq:das15} equals \eqref{eq:das16}.
													
	Next we verify $\delma \delmm \phi = \delaa \delma \phi$. For this we compute
	\begin{align*}
		& \delaa \delma \phi  (x_1, \dots,x_{n+1})  = \alpha^n(x_1) (\delma \phi)(x_2, \dots,x_{n+1})                                                                                                  \\
		                   & + \sum_{i=1}^n (-1)^i (\delma \phi) ( \alpha(x_1), \dots, x_i x_{i+1}, \dots , \alpha(x_n))                                                                              
		+ (-1)^{n+1} (\delma \phi) (x_1,\dots,x_n) \alpha^n(x_{n+1})                                            \\
		                   & = \alpha^n(x_1) \alpha(\phi(x_2,\dots,x_{n+1})  - \alpha^n(x_1) \phi(\alpha(x_2, \dots, \alpha(x_{n+1})                                                                  \\ 
		                   & + \sum_{i=1}^n (-1)^i \alpha(\phi(\alpha(x_1),\dots,x_i x_{i+1}, \dots, \alpha(x_{n+1}))                                                                                 \\
		                   & - \sum_{i=1}^n (-1)^i \phi(\alpha^2(x_1),\dots,\alpha(x_i x_{i+1}), \dots, \alpha^2(x_{n+1})                                                                             \\
		                   & + (-1)^{i+1} \alpha(\phi(x_1,\dots,x_n)) \alpha^n(x_{n+1})  - (-1)^{i+1} \phi(\alpha(x_1),\dots,\alpha(x_n)) \alpha^n(x_{n+1}) 
	\end{align*}  
	and 
	\begin{align*}
		\delma \delmm \phi & (x_1,\dots, x_{n+1})  = \alpha((\delmm \phi)(x_1, \dots, x_{n+1}) - (\delmm \phi)(\alpha(x_1),\dots,\alpha(x_{n+1})) \\
		& =\alpha(\alpha^{n-1}(x_1) \phi(x_2,\dots,x_{n+1}))                                                                   
		+ \sum_{i=1}^n (-1)^i \alpha(\phi(\alpha(x_1),\dots,x_i x_{i+1}, \dots, \alpha(x_{n+1}))              \\
		& + (-1)^{i+1} \alpha(\phi(x_1,\dots,x_n) \alpha^{n-1}(x_{n+1}))                             - \alpha^n(x_1) \phi(\alpha(x_2, \dots, \alpha(x_{n+1})         \\ 
		  & - \sum_{i=1}^n (-1)^i \phi(\alpha^2(x_1),\dots,\alpha(x_i)\alpha( x_{i+1}), \dots, \alpha^2(x_{n+1})) \\ & - (-1)^{i+1} \phi(\alpha(x_1),\dots,\alpha(x_n)) \alpha^n(x_{n+1}) .     
	\end{align*}
	Using the multiplicativity it is easily seen that the two sides agree.
	This completes the proof.
												
\end{proof}

\begin{defn}
	We denote by $\aHB^\bullet(A)=  \del \aHC^{\bullet-1}(A)$ the coboundaries and by $\aHZ(A)= \{ (\phi,\psi) \in \aHC(A) | \del(\phi,\psi) =0  \}$ the cocycles.
	The cohomology of $(\aHC(A),\del)$ is  $ \factor{\aHC(A)}{\aHB(A)}$ and we call it $\alpha$-type Hochschild cohomology of $A$ with value in itself and denote it by $\aHH(A)$.
\end{defn}

Similarly one can define the cohomology $\aHH(A,M)$ with values in an $A$-bimodule $(M,\alpha_M)$. We denote the left and right action by $\cdot$. 
In this case we have:
\begin{equation}
	\aHC^n(A,M) = \aHC^n_\mu(A,M) \oplus \aHC^n_\alpha(A,M) = \Hom(A^{\otimes n}, M) \oplus \Hom(A^{\otimes (n-1)},M)
\end{equation}
for $n \geq 2$, $\aHC^1(A,M)= \Hom(A,M)$, $\aHC^n(A,M) = 0$ for $n \leq 0$  and differentials 
\begin{align}
	\begin{split}
	\delmm \phi (x_1, \ldots, x_{n+1}) & =\alpha^{n-1}(x_1) \cdot \phi(x_2, \ldots, x_{n+1})                                                                       \\
	                                   & + \sum_{i=1}^{n} (-1)^{i} \phi(\alpha(x_1), \dots, x_i x_{i+1}, \dots , \alpha(x_{n+1}))                                  \\
	                                   & + (-1)^{n+1} \phi(x_1, \ldots, x_{n-1}) \cdot\alpha^{n-1}(x_{n+1})                                                        
	\end{split}, \\
	\begin{split}
	\delaa \psi (x_1,\ldots, x_n)      & = \alpha^{n-1}(x_1) \cdot \psi(x_2, \ldots, x_n)                                                                          \\
	                                   & + \sum_{i=1}^{n-1} (-1)^{i} \psi(\alpha(x_1), \dots, x_i x_{i+1}, \dots , \alpha(x_n))                                    \\
	                                   & + (-1)^{n} \psi(x_1, \ldots, x_{n-1})\cdot \alpha^{n-1}(x_n),                                                             
	\end{split} \\
	\delma \phi (x_1,\ldots, x_n)      & = \alpha_M( \phi(x_1,\ldots, x_n) ) - \phi(\alpha(x_1),\ldots,\alpha( x_n)),                                              \\
	\delam \psi (x_1,\ldots, x_{n+1})  & = \alpha^{n-2}(x_1 x_2) \cdot  \psi(x_3, \ldots, x_{n+1})  - \psi(x_1, \ldots, x_{n-1})\cdot \alpha^{n-2}(x_{n} x_{n+1}). 
\end{align}
	
With this we set
\begin{align}
	\del (\phi + \psi) & = (\delmm + \delma) \phi - (\delam + \delaa )\psi          \\
	                   & = ( \delmm \phi - \delam \psi, \delma \phi - \delaa \psi). 
\end{align}
	
The proof for $\del^2 = 0$ is completely analog to the proof of  \Cref{th:delalg}.

\subsection{Relationships with  other Hochschild-type cohomologies}

We  establish the connection with the cohomology given in \citep{homcoho}. For this we consider only elements where the summand in $\aHC_\alpha$ is zero, this means pairs of the form $(\phi,0)$.  
The condition $\del (\phi,0) =0$, corresponds to $\delmm \phi =0$ and $\delma \phi =0$, since the other two parts vanish. 
Since $\delma \phi \in \aHC(A)_\alpha$ and we want this part to be zero, we consider only the subcomplex where $\delma$ vanishes, this is $\HC^n_\alpha(A) = \{\phi \in \aHC^n_\mu(A) | \alpha \phi = \phi \alpha^{\otimes n} \}$. The remaining map $\delmm$ is a differential on this. In fact we have 
	
\begin{prop}
	The cohomology of $\HC^\bullet_\alpha(A)$ with differential $\delmm$ is the one given in \citep{homcoho}.
\end{prop}
\begin{proof}
	This follows directly since the complexes and the differentials agree by definition.
\end{proof}

As stated above in the associative case the cohomology is completely determined by the Hochschild cohomology. In fact we have the following
	
\begin{theorem}\label{th:cohoass}
	Let $A$ be an associative algebra, then $\aHH^k(A) \cong \HH^k(A) \oplus \HH^{k-1}(A)$ for $k \geq 1$  and $\HH(A)^1$ set to $\Der(A)$.
\end{theorem}
\begin{proof}
	Since $\alpha = \id$, we have $\delma=0$, $\delaa^{k+1} = \delmm^k$ and $\delmm$ agrees with the Hochschild differential. We prove the statement by induction over $k$.  For $k=1$ the statement is clear.
	For $(\phi,\psi) \in \aHC(A)$ to be closed, we must have $\delaa \psi =0$, so it is a cocycle in the ordinary Hochschild cohomology $\HC^{k-1}(A)$ and so cohomologous to an element in $\HH(A)$. W.l.o.g. we can assume $\psi$ is this element. 
	Then $\tilde{\phi} \in \aHC_\mu^k(A)$ defined by $\tilde\phi(x_1,\dots,x_n) = x_1 \psi(x_2,\dots,x_n) + (-1)^n \psi(x_1,\dots,x_{n-1})x_n$ satisfies $\delmm \tilde{\phi} = \delam \psi$, so $(\tilde{ \phi} ,\psi)$ is closed. Furthermore  $\delmm(\tilde{\phi} - \phi) =0$, so it is cohomologous to an element in  $\HH^k(A)$. 
\end{proof}

Now, we consider the case where $\alpha$ is invertible, in this case the construction in the previous proof can be generalized.
For a linear map  $\psi: A^{\otimes (n-1)} \to  A$, we  define  $\phi_\psi(x_1, \dots, x_n) = \alpha^{n-2}(x_1) \alpha^{-1}(\psi(x_2, \dots, x_n) + (-1)^n \alpha^{-1}\psi(x_1, \dots ,x_{n-1}) \alpha^{n-2}(x_n)$.
	
\begin{prop}
	Let $\psi: A^{\otimes (n-1)} \to  A$ be a linear map, which satisfies $\delaa \psi  = \delma \phi_\psi$. Then $(\phi_\psi, \psi)$ is a $n$-cocycle of $A$.
\end{prop}
\begin{proof}
	We consider $\xi(x_1, \dots, x_n) = \alpha^{n-2}(x_1) \alpha^{-1}(\psi(x_2, \dots, x_n)) $ and calculate $\delmm \xi$.
	\begin{align*}
		(\delmm& \xi)(x_1, \dots, x_{n+1} )  = \alpha^{n-1}(x_1) \big( \alpha^{n-2}(x_2) \alpha^{-1} \psi(x_3, \dots, x_{n+1}) \big) \\
		  & + \alpha^{n-2}(x_1 x_2) \alpha^{-1} \psi(\alpha(x_3), \dots, \alpha( x_{n+1})) \\ &+
		\sum_{i=2}^n (-1)^i \alpha^{n-1} (x_1) \alpha^{-1} \psi( \alpha(x_2), \dots, x_i x_{i+1}, \dots , \alpha(x_{n+1}))  \\
		& + (-1)^{n+1} \left( \alpha^{n-2} \alpha^{-1}\psi(x_2, \dots, x_{n}) \right) \alpha^{n-1} (x_{n+1}) \\
		&= (\alpha^{n-2}(x_1 ) \alpha^{n-2}) \psi(x_3, \dots, x_{n+1}) 
		- \alpha^{n-1} (x_1) \left( \alpha^{n-2}(x_2) \alpha^{-2}\psi(\alpha(x_3),\dots,\alpha(x_{n+1})\right)  \\
		&+ \alpha^{n-1}(x_1) \alpha^{-1}( \alpha^{n-1}(x_2) \psi(x_3,\dots,x_{n+1})) 
		- \alpha^{n-1}(x_1) \alpha^{-1}( \alpha^{n-1}(x_2) \psi(x_3,\dots,x_{n+1})) \\
		& +\sum_{i=2}^n (-1)^i \alpha^{n-1} (x_1) \alpha^{-1} \psi( \alpha(x_2), \dots, x_i x_{i+1}, \dots , \alpha(x_{n+1}))    \\ 
		&+ \alpha^{n-1}(x_1) \alpha^{-1} \left( \psi(x_1, \dots, x_{n}) \alpha^{n-1} (x_n)\right) \\
		& = \alpha^{n-2}(x_1 x_2) \psi(x_3, \dots, x_{n+1})   \\
		&+ \alpha^{n-1}(x_1) \alpha^{-1} \left( (\delma \xi)(x_2, \dots, x_{n+1}) 
		- (\delaa \psi)(x_2, \dots, x_{n+1})\right)
	\end{align*}
	Using this it is easy to see that
	\begin{align*}
		\delmm \phi_\psi(x_1,\dots,x_{n+1}) & = \delam \psi(x_1,\dots,x_{n+1}) + \alpha^{n-1}(x_1) \alpha^{-1} \big( (\delma \phi_\psi)(x_2, \dots, x_{n+1}) \\ & - (\delaa \psi)(x_2, \dots, x_{n+1})\big).	 
	\end{align*} 
	Since we assumed $\delaa \psi =\delma \phi_\psi$, then  we get the result.
\end{proof}
	
Note that the condition in the previous theorem can be written explicitly  as 
\begin{equation}
	\begin{split}
		\alpha^{n-1}(x_1) \alpha^{-1}\psi(\alpha(x_2),\dots, \alpha(x_n)) + \sum_{i=1}^{n-1}(-1)^{i} \psi(\alpha(x_1), \dots, x_i x_{i+1}, \dots, \alpha(x_n))\\  + (-1)^{n} \alpha^{-1} \psi(\alpha(x_1), \dots, \alpha(x_{n-1}) \alpha^{n-1}(x_n) = 0.
	\end{split}
\end{equation}
If $\psi$ commutes with $\alpha$, this reduces to $\delaa \psi =0$. For $n=2$ the condition  means that $\psi$ is a conjugate $\alpha$-derivation.
	
Note that in the  associative case, i.e. $\alpha = \id$, every cocycle is cohomologous to a sum of two cocycles, where ones has the form $(\phi , 0)$ and the other one the form $ (\phi_\psi, \phi)$. It would be interesting to know whether this is always the case if $\alpha$ is invertible, see also \cref{sc:yautwist}.
	

\subsection{$L_\infty$-structure} \label{Section3.2}

It would be nice to have an $L_\infty$-structure on the complex $\aHC(A)$, for a vector space $A$, such that the Maurer-Cartan elements are precisely the Hom-associative algebra, and the differential is given as usual in this context. It is clear that on elements of the form $(\phi,0)$ this should reduce to the Gerstenhaber structure given in \citep{homcoho}. It is also clear that since the equations for the Hom-associativity and mutliplicativity are not binary that  it cannot be an ordinary Lie algebra but must be a true $L_\infty$-structure.  Unfortunately we do not know such a structure, but one can try to construct it degree by degree. Using a computer, we were able to  do this for the low degrees, up to degree 5. Since the terms become quite long we only give the terms needed to do deformation theory here. But we first state the following conjecture:
	
\begin{conjecture}
	There is an $L_\infty$-algebra structure on the complex $\aHC(A)$, such that the Maurer-Cartan elements are precisely multiplicative Hom-associative algebras on $A$, and the differential defined above is induced from it.
\end{conjecture}
	
We give the brackets with values in degree up to 2, where the degree here is shifted such that $\deg \aHC^n = n-2$, so the multiplication and structure map are of degree zero.  With $\phi_i \in \aHC_\mu^i, \psi_i \in \aHC_\alpha^i, \alpha_i \in \aHC_\alpha^2$ and $\mu_i \in \aHC_\mu^2$ we have:
\begin{align*}
	\deg 1: \\
	[\mu_1 ,\mu_2, \alpha]_\mu                 & = \mu_1 (\mu_2 \otimes \alpha) -\mu_1 ( \alpha \otimes \mu_2)                                                                                                      
	+ \mu_2 (\mu_1 \otimes \alpha) -\mu_2 ( \alpha \otimes \mu_1) \\
	[\mu,\alpha]_\alpha                        & = \alpha \mu                                                                                                                                                       \\
	[\mu, \alpha_1,\alpha_2]_\alpha            & =  - \mu( \alpha_1 \otimes \alpha_2 + \alpha_2 \otimes \alpha_1)                                                                                                   \\
	\deg 2: \\
	[\phi_3, \mu,\alpha_1,\alpha_2]_\mu        & =  \phi_3 (\mu \otimes \alpha_1 \otimes \alpha_2) - \phi_3 (\alpha_1 \otimes\mu \otimes  \alpha_2) +\phi_3 (\alpha_1  \otimes  \alpha_2\otimes \mu )               \\
	                                           & + \phi_3 (\mu \otimes \alpha_2 \otimes \alpha_1) - \phi_3 (\alpha_2 \otimes\mu \otimes  \alpha_1) +\phi_3 (\alpha_2  \otimes  \alpha_1\otimes \mu )                \\
	                                           & - \mu( \alpha_1 \alpha_2 \otimes \phi_3)  - \mu(\phi_3 \otimes \alpha_1 \alpha_2 )                                                                                 
	- \mu( \alpha_2 \alpha_1 \otimes \phi_3)  - \mu(\phi_3 \otimes \alpha_2 \alpha_1 ) \\
	[\psi_2,\mu_1,\mu_2,\alpha_1,\alpha_2]     & =  \sum_{\sigma, \tau \in S_2} \mu_{\sigma1}(\psi_2 \otimes \mu_{\sigma 2}(\alpha_{\tau 1} \otimes \alpha_{\tau 2}))                                               
	-  \mu_{\sigma 1}( \mu_{\sigma 2}(\alpha_{\tau 1} \otimes \alpha_{\tau 2}) \otimes \psi_2)  \\
	[\phi_2,\alpha]_\alpha                     & = \alpha \phi_3                                                                                                                                                    \\
	[\phi_2,\alpha_1,\alpha_2,\alpha_3]_\alpha & =  - \sum_{\sigma \in S_3}  \phi_3 (\alpha_{\sigma(1)} \otimes \alpha_{\sigma(2)} \otimes \alpha_{\sigma(3)} )                                                     \\
	[\psi_2,\mu,\alpha]_\alpha                 & = \psi_2 ( \alpha \otimes \mu) - \psi_2 ( \mu \otimes \alpha)                                                                                                      \\
	[\psi_2,\mu,\alpha_1,\alpha_2]_\alpha      & = \mu ( \alpha_1 \alpha_2 \otimes \psi_2) +\mu ( \alpha_2 \alpha_1 \otimes \psi_2) - \mu(\psi_2 \otimes \alpha_1 \alpha_2) - \mu(\psi_2 \otimes \alpha_2 \alpha_1) 
\end{align*}
	
It is well known that a $L_\infty$-algebra on a graded vector  space $V$ can be given by a coderivation $\bar l$ on the graded symmetric algebra $S (V[1])$, which squares to zero.  The derivation $\bar l$ is completely determined by its corestriction to $V$, which we will denote by $l: S(V[1]) \to V[1]$. In our case we have $V= \aHC(A)$, and $l$ is given by the brackets defined above.
	
To prove that this is in fact a $L_\infty$-structure, it is easiest to consider a graph complex. This is based on the approach for example in \citep{MR2812919} for defining $L_\infty$ structure governing deformations if one knows a model  for the corresponding operad. We refer to previous reference for details. Another way to see this is using the fact that a codifferential on the cofree conilpotent cooperad gives a $L_\infty$-structures as we want it here, see \citep[Section 10.5]{MR2954392}. To use this one has to use the weight graded dual of the free operad  given here, i.e. graded by the numbers of generated.  
We only give the general ideas here since, we only need it as a tool to motivate the brackets defined above and to show that they form in fact an $L_\infty$-structure, which one could also do by hand.
	
The graph complex consists of planar rooted trees formed by vertices \verm2, \verm3, \verm4, \vera1 \vera2 \vera3. This means, it corresponds to  the free operad generated by these operations. We consider the graphs to be graded, such that $\deg(\vermk) = k-2$ and $\deg(\verak) = k-1$. This means in particular that \verm2 and \vera1, which correspond to $\alpha$ and $\mu$, are of degree 0.
The differential is given by $ \del \verm2 = \del \vera1 = 0$, 
\begin{align*}
	\del \verm3 & = 
	\begin{tikzpicture}[scale=0.3,point/.style={draw,shape=circle,fill=blue,minimum size=2,inner sep=0}, left/.style={draw,regular polygon, regular polygon sides=3, rotate=90,minimum size=5,inner sep=0},
	right/.style={draw,regular polygon, regular polygon sides=3, rotate=-90,minimum size=5,inner sep=0},
	circ/.style={draw,shape=circle,minimum size=5,inner sep=0}]
	\node [point] (a0) at (2.25,2) {};
	\node [point] (a1) at (1.5,1) {};
	\draw (a0) -- +(0,0.7) ;
	\draw (1,0) -- (a1);
	\draw (2,0) -- (a1);
	\draw  (a1) -- (a0); 
	\node [circ] (a2) at (3,1) {};
	\draw (3,0) -- (a2);
	\draw  (a2) -- (a0); 
	\end{tikzpicture}
	-\begin{tikzpicture}[scale=0.3,point/.style={draw,shape=circle,fill=blue,minimum size=2,inner sep=0}, left/.style={draw,regular polygon, regular polygon sides=3, rotate=90,minimum size=5,inner sep=0},
	right/.style={draw,regular polygon, regular polygon sides=3, rotate=-90,minimum size=5,inner sep=0},
	circ/.style={draw,shape=circle,minimum size=5,inner sep=0}]
	\node [point] (a0) at (1.75,2) {};
	\node [circ] (a1) at (1,1) {};
	\draw (a0) -- +(0,0.7) ; 	\draw (1,0) -- (a1);
	\draw  (a1) -- (a0); 
	\node [point] (a2) at (2.5,1) {};
	\draw (2,0) -- (a2);
	\draw (3,0) -- (a2);
	\draw  (a2) -- (a0); 
	\end{tikzpicture}, \;
	\del \vera2 =
	\begin{tikzpicture}[scale=0.3,point/.style={draw,shape=circle,fill=blue,minimum size=2,inner sep=0}, left/.style={draw,regular polygon, regular polygon sides=3, rotate=90,minimum size=5,inner sep=0},
	right/.style={draw,regular polygon, regular polygon sides=3, rotate=-90,minimum size=5,inner sep=0},
	circ/.style={draw,shape=circle,minimum size=5,inner sep=0}]
	\node [circ] (a0) at (1.5,2) {};
	\node [point] (a1) at (1.5,1) {};
	\draw (a0) -- +(0,0.7) ; 	\draw (1,0) -- (a1);
	\draw (2,0) -- (a1);
	\draw  (a1) -- (a0); 
	\end{tikzpicture}
	-\begin{tikzpicture}[scale=0.3,point/.style={draw,shape=circle,fill=blue,minimum size=2,inner sep=0}, left/.style={draw,regular polygon, regular polygon sides=3, rotate=90,minimum size=5,inner sep=0},
	right/.style={draw,regular polygon, regular polygon sides=3, rotate=-90,minimum size=5,inner sep=0},
	circ/.style={draw,shape=circle,minimum size=5,inner sep=0}]
	\node [point] (a0) at (1.5,2) {};
	\node [circ] (a1) at (1,1) {};
	\draw (a0) -- +(0,0.7) ; 	\draw (1,0) -- (a1);
	\draw  (a1) -- (a0); 
	\node [circ] (a2) at (2,1) {};
	\draw (2,0) -- (a2);
	\draw  (a2) -- (a0); 
	\end{tikzpicture}, \\
	\del \verm4 & = 
	\begin{tikzpicture}[scale=0.3,point/.style={draw,shape=circle,fill=blue,minimum size=2,inner sep=0}, left/.style={draw,regular polygon, regular polygon sides=3, rotate=90,minimum size=5,inner sep=0},
	right/.style={draw,regular polygon, regular polygon sides=3, rotate=-90,minimum size=5,inner sep=0},
	circ/.style={draw,shape=circle,minimum size=5,inner sep=0}]
	\node [point] (a0) at (2.83333,2) {};
	\node [point] (a1) at (1.5,1) {};
	\draw (a0) -- +(0,0.7) ; 	\draw (1,0) -- (a1);
	\draw (2,0) -- (a1);
	\draw  (a1) -- (a0); 
	\node [circ] (a2) at (3,1) {};
	\draw (3,0) -- (a2);
	\draw  (a2) -- (a0); 
	\node [circ] (a3) at (4,1) {};
	\draw (4,0) -- (a3);
	\draw  (a3) -- (a0); 
	\end{tikzpicture}
	-\begin{tikzpicture}[scale=0.3,point/.style={draw,shape=circle,fill=blue,minimum size=2,inner sep=0}, left/.style={draw,regular polygon, regular polygon sides=3, rotate=90,minimum size=5,inner sep=0},
	right/.style={draw,regular polygon, regular polygon sides=3, rotate=-90,minimum size=5,inner sep=0},
	circ/.style={draw,shape=circle,minimum size=5,inner sep=0}]
	\node [point] (a0) at (2.5,2) {};
	\node [circ] (a1) at (1,1) {};
	\draw (a0) -- +(0,0.7) ; 	\draw (1,0) -- (a1);
	\draw  (a1) -- (a0); 
	\node [point] (a2) at (2.5,1) {};
	\draw (2,0) -- (a2);
	\draw (3,0) -- (a2);
	\draw  (a2) -- (a0); 
	\node [circ] (a3) at (4,1) {};
	\draw (4,0) -- (a3);
	\draw  (a3) -- (a0); 
	\end{tikzpicture}
	+ \begin{tikzpicture}[scale=0.3,point/.style={draw,shape=circle,fill=blue,minimum size=2,inner sep=0}, left/.style={draw,regular polygon, regular polygon sides=3, rotate=90,minimum size=5,inner sep=0},
	right/.style={draw,regular polygon, regular polygon sides=3, rotate=-90,minimum size=5,inner sep=0},
	circ/.style={draw,shape=circle,minimum size=5,inner sep=0}]
	\node [point] (a0) at (2.16667,2) {};
	\node [circ] (a1) at (1,1) {};
	\draw (a0) -- +(0,0.7) ; 	\draw (1,0) -- (a1);
	\draw  (a1) -- (a0); 
	\node [circ] (a2) at (2,1) {};
	\draw (2,0) -- (a2);
	\draw  (a2) -- (a0); 
	\node [point] (a3) at (3.5,1) {};
	\draw (3,0) -- (a3);
	\draw (4,0) -- (a3);
	\draw  (a3) -- (a0); 
	\end{tikzpicture}
	-\begin{tikzpicture}[scale=0.3,point/.style={draw,shape=circle,fill=blue,minimum size=2,inner sep=0}, left/.style={draw,regular polygon, regular polygon sides=3, rotate=90,minimum size=5,inner sep=0},
	right/.style={draw,regular polygon, regular polygon sides=3, rotate=-90,minimum size=5,inner sep=0},
	circ/.style={draw,shape=circle,minimum size=5,inner sep=0}]
	\node [point] (a0) at (2,3) {};
	\node [circ] (a1) at (1,2) {};
	\node [circ] (a2) at (1,1) {};
	\draw (a0) -- +(0,0.7) ; 	\draw (1,0) -- (a2);
	\draw  (a2) -- (a1); 
	\draw  (a1) -- (a0); 
	\node [point] (a3) at (3,1) {};
	\draw (2,0) -- (a3);
	\draw (3,0) -- (a3);
	\draw (4,0) -- (a3);
	\draw  (a3) -- (a0); 
	\end{tikzpicture}
	-\begin{tikzpicture}[scale=0.3,point/.style={draw,shape=circle,fill=blue,minimum size=2,inner sep=0}, left/.style={draw,regular polygon, regular polygon sides=3, rotate=90,minimum size=5,inner sep=0},
	right/.style={draw,regular polygon, regular polygon sides=3, rotate=-90,minimum size=5,inner sep=0},
	circ/.style={draw,shape=circle,minimum size=5,inner sep=0}]
	\node [point] (a0) at (3,3) {};
	\node [point] (a1) at (2,1) {};
	\draw (a0) -- +(0,0.7) ; 	\draw (1,0) -- (a1);
	\draw (2,0) -- (a1);
	\draw (3,0) -- (a1);
	\draw  (a1) -- (a0); 
	\node [circ] (a2) at (4,2) {};
	\node [circ] (a3) at (4,1) {};
	\draw (4,0) -- (a3);
	\draw  (a3) -- (a2); 
	\draw  (a2) -- (a0); 
	\end{tikzpicture}
	+\begin{tikzpicture}[scale=0.3,point/.style={draw,shape=circle,fill=blue,minimum size=2,inner sep=0}, left/.style={draw,regular polygon, regular polygon sides=3, rotate=90,minimum size=5,inner sep=0},
	right/.style={draw,regular polygon, regular polygon sides=3, rotate=-90,minimum size=5,inner sep=0},
	circ/.style={draw,shape=circle,minimum size=5,inner sep=0}]
	\node [point] (a0) at (2.5,3) {};
	\node [circ] (a1) at (1.5,1) {};
	\draw (a0) -- +(0,0.7) ; 	\draw (1,0) -- (a1);
	\draw (2,0) -- (a1);
	\draw  (a1) -- (a0); 
	\node [point] (a2) at (3.5,2) {};
	\node [circ] (a3) at (3,1) {};
	\draw (3,0) -- (a3);
	\draw  (a3) -- (a2); 
	\node [circ] (a4) at (4,1) {};
	\draw (4,0) -- (a4);
	\draw  (a4) -- (a2); 
	\draw  (a2) -- (a0); 
	\end{tikzpicture}
	-\begin{tikzpicture}[scale=0.3,point/.style={draw,shape=circle,fill=blue,minimum size=2,inner sep=0}, left/.style={draw,regular polygon, regular polygon sides=3, rotate=90,minimum size=5,inner sep=0},
	right/.style={draw,regular polygon, regular polygon sides=3, rotate=-90,minimum size=5,inner sep=0},
	circ/.style={draw,shape=circle,minimum size=5,inner sep=0}]
	\node [point] (a0) at (2.5,3) {};
	\node [point] (a1) at (1.5,2) {};
	\node [circ] (a2) at (1,1) {};
	\draw (a0) -- +(0,0.7) ; 	\draw (1,0) -- (a2);
	\draw  (a2) -- (a1); 
	\node [circ] (a3) at (2,1) {};
	\draw (2,0) -- (a3);
	\draw  (a3) -- (a1); 
	\draw  (a1) -- (a0); 
	\node [circ] (a4) at (3.5,1) {};
	\draw (3,0) -- (a4);
	\draw (4,0) -- (a4);
	\draw  (a4) -- (a0); 
	\end{tikzpicture}, \\
	\del \vera3 & = 
	\begin{tikzpicture}[scale=0.3,point/.style={draw,shape=circle,fill=blue,minimum size=2,inner sep=0}, left/.style={draw,regular polygon, regular polygon sides=3, rotate=90,minimum size=5,inner sep=0},
	right/.style={draw,regular polygon, regular polygon sides=3, rotate=-90,minimum size=5,inner sep=0},
	circ/.style={draw,shape=circle,minimum size=5,inner sep=0}]
	\node [circ] (a0) at (2,2) {};
	\node [point] (a1) at (2,1) {};
	\draw (a0) -- +(0,0.7) ; 	\draw (1,0) -- (a1);
	\draw (2,0) -- (a1);
	\draw (3,0) -- (a1);
	\draw  (a1) -- (a0); 
	\end{tikzpicture}
	-\begin{tikzpicture}[scale=0.3,point/.style={draw,shape=circle,fill=blue,minimum size=2,inner sep=0}, left/.style={draw,regular polygon, regular polygon sides=3, rotate=90,minimum size=5,inner sep=0},
	right/.style={draw,regular polygon, regular polygon sides=3, rotate=-90,minimum size=5,inner sep=0},
	circ/.style={draw,shape=circle,minimum size=5,inner sep=0}]
	\node [point] (a0) at (2,2) {};
	\node [circ] (a1) at (1,1) {};
	\draw (a0) -- +(0,0.7) ; 	\draw (1,0) -- (a1);
	\draw  (a1) -- (a0); 
	\node [circ] (a2) at (2,1) {};
	\draw (2,0) -- (a2);
	\draw  (a2) -- (a0); 
	\node [circ] (a3) at (3,1) {};
	\draw (3,0) -- (a3);
	\draw  (a3) -- (a0); 
	\end{tikzpicture}
	-\begin{tikzpicture}[scale=0.3,point/.style={draw,shape=circle,fill=blue,minimum size=2,inner sep=0}, left/.style={draw,regular polygon, regular polygon sides=3, rotate=90,minimum size=5,inner sep=0},
	right/.style={draw,regular polygon, regular polygon sides=3, rotate=-90,minimum size=5,inner sep=0},
	circ/.style={draw,shape=circle,minimum size=5,inner sep=0}]
	\node [circ] (a0) at (2.25,2) {};
	\node [point] (a1) at (1.5,1) {};
	\draw (a0) -- +(0,0.7) ; 	\draw (1,0) -- (a1);
	\draw (2,0) -- (a1);
	\draw  (a1) -- (a0); 
	\node [circ] (a2) at (3,1) {};
	\draw (3,0) -- (a2);
	\draw  (a2) -- (a0); 
	\end{tikzpicture}
	+ \begin{tikzpicture}[scale=0.3,point/.style={draw,shape=circle,fill=blue,minimum size=2,inner sep=0}, left/.style={draw,regular polygon, regular polygon sides=3, rotate=90,minimum size=5,inner sep=0},
	right/.style={draw,regular polygon, regular polygon sides=3, rotate=-90,minimum size=5,inner sep=0},
	circ/.style={draw,shape=circle,minimum size=5,inner sep=0}]
	\node [circ] (a0) at (1.75,2) {};
	\node [circ] (a1) at (1,1) {};
	\draw (a0) -- +(0,0.7) ; 	\draw (1,0) -- (a1);
	\draw  (a1) -- (a0); 
	\node [point] (a2) at (2.5,1) {};
	\draw (2,0) -- (a2);
	\draw (3,0) -- (a2);
	\draw  (a2) -- (a0); 
	\end{tikzpicture}
	+\begin{tikzpicture}[scale=0.3,point/.style={draw,shape=circle,fill=blue,minimum size=2,inner sep=0}, left/.style={draw,regular polygon, regular polygon sides=3, rotate=90,minimum size=5,inner sep=0},
	right/.style={draw,regular polygon, regular polygon sides=3, rotate=-90,minimum size=5,inner sep=0},
	circ/.style={draw,shape=circle,minimum size=5,inner sep=0}]
	\node [point] (a0) at (1.75,3) {};
	\node [circ] (a1) at (1,2) {};
	\node [circ] (a2) at (1,1) {};
	\draw (a0) -- +(0,0.7) ; 	\draw (1,0) -- (a2);
	\draw  (a2) -- (a1); 
	\draw  (a1) -- (a0); 
	\node [circ] (a3) at (2.5,1) {};
	\draw (2,0) -- (a3);
	\draw (3,0) -- (a3);
	\draw  (a3) -- (a0); 
	\end{tikzpicture}
	-\begin{tikzpicture}[scale=0.3,point/.style={draw,shape=circle,fill=blue,minimum size=2,inner sep=0}, left/.style={draw,regular polygon, regular polygon sides=3, rotate=90,minimum size=5,inner sep=0},
	right/.style={draw,regular polygon, regular polygon sides=3, rotate=-90,minimum size=5,inner sep=0},
	circ/.style={draw,shape=circle,minimum size=5,inner sep=0}]
	\node [point] (a0) at (2.25,3) {};
	\node [circ] (a1) at (1.5,1) {};
	\draw (a0) -- +(0,0.7) ; 	\draw (1,0) -- (a1);
	\draw (2,0) -- (a1);
	\draw  (a1) -- (a0); 
	\node [circ] (a2) at (3,2) {};
	\node [circ] (a3) at (3,1) {};
	\draw (3,0) -- (a3);
	\draw  (a3) -- (a2); 
	\draw  (a2) -- (a0); 
	\end{tikzpicture}.
\end{align*}
It is an easy calculation to show that $\del$ squares to zero.

There is a pairing between a graph and an element  $\phi_1 \cdot  \dots \cdot \phi_k \cdot\psi_1 \cdot \dots  \cdot \psi_l$  of the graded symmetric algebra  $S(\aHC)$ with values in $\aHC(A)$. It is given on the generators by $\sprod{\vermk , \phi} = \phi$ if $\phi \in \aHC^k_\mu$, $\sprod{\verak, \psi}$ if $\psi \in \aHC^{k-1}_\alpha$. For general graphs it is given as the sum over all permutations of possible assignments. Note that since the objects are graded this includes signs.
This results in a tree with each element decorated by an element of $\aHC$ this can be made to an element in $\aHC$ again by composition as the tree describes.  
	
The $L_\infty$-structure is given by $[\psi_1,\dots,\phi_1,\dots]_\mu = \sprod{ \del \vermk , \phi_1 \cdot \phi_k \dots \cdot \psi_1 \cdot \dots  \cdot \psi_l }$ and $[\psi_1,\dots,\phi_1,\dots]_\alpha = \sprod{ \del \verak ,\phi_1 \cdot \phi_k \dots \cdot \psi_1 \cdot \dots  \cdot \psi_l}$, for the restriction of the bracket to $\aHC(A)^k_\mu$ and $\aHC(A)^{k+1}_\alpha$ respectively.  
	
The bracket can be extended to a coderivation $\bar l$  of $S(\aHC)$, and it is a $L_\infty$-structure if and only if $l^2 =0$.
The fact that $\bar l^2 =0$ follows directly from $\del^2=0$.
	
\begin{prop}
	A Maurer-Cartan element on this $L_\infty$-algebra is a Hom-associative algebra, and for a Hom-associative algebra structure $(\mu,\alpha)$ the differential in degree two and three on $\aHC$ are given by $\del (\phi,\psi) = l(\E^{(\mu,\alpha)} (\phi,\psi))$.
\end{prop}
\begin{proof}
	The bracket defined by graphs are the same as the brackets given above. It is clear by looking at the defining equations that a Maurer-Cartan element for them is a Hom-associative algebra and the differential in degree two and three are given by these brackets as $\del (\phi,\psi) = [\mu, \dots,\alpha, \dots ,\phi]_\mu + [\mu, \dots,\alpha, \dots ,\psi]_\mu , [\mu, \dots,\alpha, \dots ,\phi]_\alpha + [\mu, \dots,\alpha, \dots ,\psi]_\alpha)$, $\mu,\dots,\alpha, \dots$ stands for zero or more insertions of $\alpha$ and $\mu$. This can be written as $\del (\phi,\psi) = l (\E^{(\mu,\alpha)} \cdot (\phi,\psi))$.
\end{proof}

\subsection{$\alpha$-type cohomology under Yau twist} \label{sc:yautwist}
	
We study now  the relation between two Hom-associative algebras related by a Yau twist. 
	
Let $A$ be a Hom-associative algebra and $\gamma$ a Hom-algebra morphism of $A$. We consider the Hom-associative algebra $A_\gamma$  obtained by Yau twist.
	
\begin{prop}
	If $(\phi,\psi)$ is a $n$-cocycle of $A$, which commutes with $\gamma$, i.e. $\gamma \phi = \phi  \gamma^{\otimes n}$ and $\gamma \psi = \psi \gamma^{\otimes (n-1)}$, then $(\tilde{\phi},\tilde{\psi}) =(\gamma^{n-1} \phi, \gamma^{n-1} \psi)$ is a $n$-cocycle of $A_\gamma$.
\end{prop}
\begin{proof}
	We show that $\tilde\del (\tilde{\phi},\tilde{\psi}) = \gamma^n ( \del(\phi, \psi))$, where $\del$ denotes the differential in $A$ and $\tilde{\del}$ the one in $A_\gamma$.
	This is an easy calculation, we show it for one of the differentials, the others work analogously.
	\begin{align*}
		(\delaa \tilde\psi) (x_1, \dots, x_n) & =  \tilde{\alpha}^{n-1}(x_1) \tilde{\cdot} \tilde{\psi}(x_2, \dots, x_n) + \sum_i (-1)^i \tilde{\psi}(\tilde{\alpha}(x_1) \dots,x_i \tilde{\cdot} x_{i+1}, \dots, \tilde{\alpha}(x_n) \\ & + \tilde{\psi}(x_1, \dots, x_{n-1}) \tilde{\cdot} \tilde{\alpha}^{n-1}(x_n) \\
		&= \gamma \left( \gamma^{n-1}(\alpha^{n-1}(x_1)) \gamma^{n-1} (\psi(x_2,\dots, x_n)) \right) \\
		&+ \sum (-1)^i  \gamma^{n-1}\big(\psi\big(\gamma(\alpha(x_1)), \dots, \gamma(x_i x_{i+1}), \dots, \gamma(\alpha(x_n))\big)\big)  \\
		&+ \gamma\left( \gamma^{n-1}( \psi(x_1, \dots, x_{n-1})) \gamma^{n-1}( \alpha^{n-1}(x_n)) \right)   \\
		&= \gamma^n( (\delaa \psi)( x_1,\dots, x_n) ).
	\end{align*}
\end{proof}

We aim to study the  $\alpha$-type cohomology of Hom-associative algebras of associative type. This includes in particular the case where $\alpha$ is invertible. 
For this we first need another complex, which gives a cohomology for an associative algebra and an endomorphisms on it. The cohomology of a homomorphism between two different algebras is well known and was introduced in \citep{gerstenhabermor1,gerstenhabermor2}. More recently a $L_\infty$-structure on this complex was found in \citep{fregier1,fregier2}. 
To consider the cohomology of an endomorphism  one just has to consider the complex, where the cochains corresponding to the different algebras agree.
The complex for an associative algebra $A$ and an endomorphism $\gamma$ is given by 
\begin{equation*}
	C^n(\gamma) :=	C^n(A,\gamma) :=  C^n_\mu(A,\gamma) \oplus C^n_\gamma(A,\gamma) = \Hom(A^{\otimes n},A) \oplus \Hom(A^{\otimes (n-1)},\tilde A),
\end{equation*}
for $n \geq 2$, $C^1(A,\gamma):=\Hom(A,\tilde A)$ and $C^n(A,\gamma) =0$ for $n\leq 0$, where $\tilde{A}$ is the algebra $A$ regarded as an $A$-bimodule with the action twisted by $\gamma$, i.e. the left action is given by $a \cdot b = \gamma(a)b$ for $a \in A, b \in \tilde A$ and similarly for the right action.
The differential for an element $(\phi,\psi) \in C^n(\gamma)$ is given by $\del (\phi, \psi) = ( \del_\mu \phi, -\del_\mu \psi + \del_\gamma \phi)$, where $\del_\gamma$ is given by $\del_\gamma \phi = \gamma \phi - \phi \gamma^{\otimes n}$ and $\del_\mu$ is the ordinary Hochschild differential.
	
Note that again we have used $\{0\}$ instead of $\Hom(\K,\tilde A)$ in the low degrees since this simplifies making the connection to the $\alpha$-type Hochschild cohomology.
	
This cohomology can be computed, using spectral sequences. But since the bicomplex only has two rows not much of the theory is needed. We denote by $H(C(\gamma),\del_\mu)$ the cohomology with respect to $\del_\mu$ and $\delmb$. On this there is a differential induced by $\del_\gamma$. 
With this we have 
\begin{prop}
	The cohomology of the total complex $H(C(\gamma),\del) = H(H(C(\gamma),\del_\mu),\del_\gamma)$.
\end{prop}
\begin{proof}
	This follows directly considering the spectral sequence associated to the bicomplex $C(\gamma)$ with the vertical filtration, since then the differential on the second sheet already vanishes, since the complex only has two rows. 
\end{proof}
If one knows the Hochschild cohomology of the algebra $A$ well, this can be computed quite easily.
Also note that there is a chain map $(C^\bullet_\mu(\gamma),\del_\mu) \to (C^{\bullet+1}_\gamma(\gamma),\del_{\tilde\mu})$ given by $\phi \mapsto \gamma \phi$, which is even an isomorphism if $\gamma$ is invertible. 
	
On the other hand the other iterated  cohomology $H(H(C(\gamma),\del_\gamma),\del_\mu)$ is in general  not isomorphic to the cohomology   $H(\gamma)$, since there is still a non-trivial differential on it. Only the cohomology with respect to this differential gives the total cohomology. However it turns out that this differential vanishes in nice situations.
If this is the case every cocycle is cohomologous to either one if the form $(\phi,0)$ or one of the form $(0,\psi)$.

Now let $A$ be an associative algebra, $\gamma$ an algebra map   and $A_\gamma$ the Yau twist of $A$ by $\gamma$. Then we can define a map $\Phi:C(A,\gamma) \to \aHC(A_\gamma)$ for $(\phi,\psi) \in C^n(A,\gamma)$ by 
\begin{equation}
	(\phi,\psi) \mapsto ( \gamma^{n-1} \phi + \sum_{i=1}^{n-1}  (-1)^i \gamma^{n-2} \psi \circ_i \mu, \gamma^{n-2} \psi), 
\end{equation}
where $\psi \circ_i \mu := \psi (\id^{\otimes i-1} \otimes \mu \otimes \id^{\otimes n-i-1})$ .
	
\begin{theorem}\label{th:cohoyau}
	The map $\Phi$ is a chain map, so it induces a map in cohomology. If $\gamma$ is invertible it is an isomorphism and especially the corresponding cohomologies are also isomorphic.
\end{theorem}
\begin{proof}
	The fact that it is a chain map is a straightforward calculation. 
												
	If $\gamma$ is invertible, the inverse of $\Phi$ is given by
	\begin{equation}
		(\phi,\psi) \mapsto ( \gamma^{-n+1} \phi -  \sum_{i=1}^{n-1} (-1)^i (\gamma^{-n+1} \psi) \circ_i \mu , \gamma^{-(n-2)} \psi),
	\end{equation}
	which is easy to check.
\end{proof}

Using this proposition and the previous remarks on how to compute $H(\gamma)$ it is possible to compute the $\alpha$-type Hochschild cohomology of $A_\gamma$, if one knows the Hochschild cohomology of $A$. Especially we have
\begin{corollary}
	If $(A,\mu,\alpha)$ is a Hom-associative algebra, such that $\alpha$ is invertible, then its $\alpha$-type cohomology $\aHH(A)$ is isomorphic to $H(A_{\alpha^{-1}},\alpha)$, where $A_{\alpha^{-1}}$ is the associative algebra obtained by Yau twist of $A$ with $\alpha^{-1}$.  
\end{corollary}
\begin{corollary}
	Let $(A,\mu,\alpha)$ be a Hom-associative algebra of associative-type and $\tilde A$ the corresponding associative algebra, such that $\tilde A_\alpha =A$.  Then $H(\tilde A,\alpha)$ is a subspace of $\aHH(A)$ via $\Phi$.
\end{corollary}

\section{Deformation theory including the structure map $\alpha$}\label{sc:deformation}
	
In this section we consider deformations of  Hom-associative algebras and their  descriptions using the cohomology defined in the previous section. Up to now mostly the case where only the multiplication is changed is discussed. But here we want to consider the more  general situation where the multiplication as well as the structure map are deformed. We study the deformation equation and describe infinitesimal deformations. Moreover, we discuss obstructions to extend a deformation of  order $n$ to a deformation of order $n+1$ and  show that they gives rise to 3-cocycles in the  $\alpha$-type Hochschild cohomology. We also give a generalization of the well known fact that a deformation of a commutative algebra gives rise to a Poisson algebra.
	
First we want to recall the basic definitions of formal deformations using formal power series as introduced by Gerstenhaber in  \citep{gerstenhaber1}. 
	
Let $\K\ph$ be a formal power series ring and $A\ph$ be a formal power series space whose element are of the form $\sum_{k=0}^{\infty} a_k t^k$ with $a_k\in A$.
	
\begin{defn}
	Let $(A,\mu,\alpha)$ be a Hom-associative algebra over $\K$, then a deformation of $A$ is a Hom-associative algebra $(A\ph, \star, \alpha_\star)$ over $\K\ph$, where $a \star b = \mu_\star(a,b) = \sum_{i=0}^\infty t^i \mu_i(a,b)$ and $\alpha_\star = \sum_{i=0}^\infty  t^i \alpha_i$, with $\mu_i:A\otimes A\rightarrow A$ and $\alpha_i:A\rightarrow A$  $\K$-linear maps, such that $ a \star b =a b  \mod t$ for all $a,b  \in A$ and $\alpha_\star = \alpha \mod t$. 
\end{defn}
	
In the following we set  $\alpha_0 = \alpha$ and $\mu(a,b) =\mu_0(a,b) = ab$. We say that a deformation is of order $n$ if $a \star b =  \sum_{i=0}^n t^i \mu_i(a,b)$ and $\alpha_\star = \sum_{i=0}^n  t^i \alpha_i$ such that $\mu_n$ or $\alpha_n$ is nonzero.

\begin{defn}
	Two deformations $(A\ph,\star, \alpha_\star)$ and $(A\ph,\star', \alpha'_\star)$ of a Hom-associative algebra $A$ are said equivalent if there exists a  $\K\ph$-linear map $T$ on $A\ph$ of the form $T = \id + \sum_{i=0}^\infty T_i t^i$, such that for all $a,b \in A$, we have  $$T(a \star' b) = T(a) \star T(b) \text{ and } T(\alpha_\star'(a)) = \alpha_\star(T(a)).$$ 
\end{defn}
	
Given a deformation of $A$ and an isomorphism $T$ as above, one can also define an equivalent deformation by $a \star' b = T^{-1}(T(a) \star T(b))$ and $\alpha_\star'(a) = T^{-1}(\alpha_\star(T(a)))$. It is clear that  equivalence of deformations  is an equivalence relation on the set of deformations of a Hom-associative algebra.

Let $(A\ph, \star, \alpha_\star)$ be a deformation of the Hom-associative algebra $(A,\mu,\alpha)$, where $a \star b = \mu_\star(a,b) = \sum_{i=0}^\infty t^i \mu_i(a,b)$ and $\alpha_\star = \sum_{i=0}^\infty  t^i \alpha_i$. 
	
The Hom-associativity condition $\mu_\star(\mu_\star(a,b),\alpha_\star (c)-\mu_\star(\alpha_\star (a),\mu_\star(b,c))$ can be written as 
\begin{equation}
	\sum_{i,j,k =0 }^\infty t^{i+j+k}(\mu_i(\alpha_j(a), \mu_k(b,c)) -  \mu_i(\mu_k(a,b), \alpha_j(c)) )= 0
\end{equation}
and is equivalent to an infinite system of equations, called deformation equation with respect to Hom-associativity,   such that the $n^{th}$ equation is of the form
$$ \sum_{\substack{i,j,k\geq 0 \\ i+j+k=n}}  \mu_i(\alpha_j(a), \mu_k(b,c)) -  \mu_i(\mu_k(a,b), \alpha_j(c) ) =0.$$
Notice that the $0^{th}$ equation expresses the Hom-associativity of $A$.
	
Now, rearranging the terms and using coboundary operators from Section \ref{Section3}, one may write the previous equation as 	
\begin{equation}
	\begin{split}
		(\delmm \mu_n + \delam \alpha_n)(a,b,c) &= 
		\sum_{\substack{i,j,k=0,\dots, n-1 \\ i+j+k=n}}  \mu_i(\alpha_j(a), \mu_k(b,c)) -  \mu_i(\mu_k(a,b), \alpha_j(c) ), \label{eq:assdef}
	\end{split}	
\end{equation}
where  $   \delmm \mu_n(a,b,c)= \alpha_0(a) \mu_n(b,c)-\mu_n(ab,\alpha_0(c)) +\mu_n(\alpha_0(a), bc) -   \mu_n(a,b)\alpha_0(c)  $ and  $\delam \alpha_n(a,b,c)= \alpha_n(a) (bc) - (ab)\alpha_n(c) $.

Similarly the multiplicativity condition $\alpha_\star( \mu_\star(a,b)) = \mu_\star(\alpha_\star(a), \alpha_\star(b))$ can be written as 
\begin{equation}
	\sum_{i,j =0}^{\infty} t^{i+j}\alpha_i(\mu_j(a,b)) - 	\sum_{i,j,k =0}^{\infty} t^{i+j+k} \mu_i(\alpha_j(a), \alpha_k(b)) =0.
\end{equation}
This again is equivalent to an infinite system of equations, called deformation equation with respect to multiplicativity, with the $n^{th}$ equation given by
\begin{equation}
	\sum_{\substack{i,j =0,\dots,n \\ i+j =n}} \alpha_i(\mu_j(a,b)) -  \sum_{\substack{i,j,k =0,\dots,n \\ i+j+k =n}}  \mu_i(\alpha_j(a), \alpha_k(b)) =0,
\end{equation}
which can be rearranged to 
\begin{equation}
	\begin{split}
		(\delaa  \alpha_n - \delam \mu_n)(a,b) 
		 = \sum_{\substack{i,j =0,\dots,n-1 \\ i+j =n}} \alpha_i(\mu_j(a,b)) -  \sum_{\substack{i,j,k =0,\dots,n-1 \\ i+j+k =n}}  \mu_i(\alpha_j(a), \alpha_k(b)) , \label{eq:multdef}
	\end{split}	
\end{equation}
where $\delaa  \alpha_n(a,b)= \alpha_0(a) \alpha(b) - \alpha_n(ab) - \alpha_n(a) \alpha_0(b) $ and  $ \delam \mu_n(a,b)=\alpha_0(\mu_n(a,b)) - \mu_n(\alpha_0(a),\alpha_0(b))$.

We denote by $R_n^1$ the right hand side of equation \eqref{eq:assdef} and $R_n^2$ the right hand side of equation \eqref{eq:multdef}. The pair $(R_n^1,R_n^2)$ is called the $n^{th}$ obstruction.

Since the deformation is governed by a $L_\infty$-algebra, we have the usual statement relating deformations and cohomology.

\begin{theorem}
	Let $(A,\mu,\alpha)$ be a Hom-associative algebra and $(A\ph,\star,\alpha_\star)$ be a deformation of $A$.  Then we have 
	\begin{enumerate}
	\item  The first order term of the deformation  is a 2-cocycle, i.e. we have 
	$\del( \mu_1, \alpha_1) = 0$,  and its cohomology class is invariant under equivalence. 
	\item The $n^{th}$ deformation equations with respect to Hom-associativity and multiplicativity respectively,  are equivalent to  $\del(\mu_n,\alpha_n) = (R^1_n,R^2_n) $. Moreover,  $(R^1_n,R^2_n)$ is a 3-cocycle, i.e. $\del (R^1_n,R^2_n) = 0$. 
	\end{enumerate}
\end{theorem}
\begin{proof}
	The equation $\del(\mu_n,\alpha_n) = (R^1_n,R^2_n) $ is equivalent to \cref{eq:multdef,eq:assdef}.
	For $n=1$ \cref{eq:assdef}  gives $(\delmm \mu_n + \delam \alpha_n)(a,b,c) = 0$ and \cref{eq:multdef} gives  $(\delaa  \alpha_n - \delam \mu_n)(a,b)=0$, since obviously the right hand side vanishes.  This means we have $\del(\mu_1,\alpha_1) = 0$.  
	
	For an equivalent deformation $(\star',\alpha_\star')$, we get from $S(a \star b )= S(a) \star S(b)$ in first order that $\mu'_1(a,b) +S_1(ab) = S_1(a) b + a S_1(b) + \mu_1(a,b)$. Further from $ S( \alpha'_\star(a)) = \alpha_\star(S)(a)$, we get $S_1(\alpha_0(a)) + \alpha'_1(a) = \alpha_1(a) + \alpha_0 S_1(a)$. Rearranging  gives $(\mu'_1 - \mu_1, \alpha'_1 -\alpha_1 ) =  \del ( S,0)$. So the cohomology class of $(\mu_1,\alpha_1)$ and $(\mu'_1,\alpha'_1)$ are the same as claimed.

	Using the $L_\infty$-structure $l$ defined in Section \ref{Section3.2}, we have $ \del (\mu_n,\alpha_n)  -R_n  = l(\E^{(\mu_\star,\alpha_\star)})$ at order $n$ in $t$, since  both are equivalent to the fact that $(\mu_\star,\alpha_\star)$ is a Hom-associative algebra.  Also $l(\E^{(\mu_\star,\alpha_\star)})$ vanishes up to order $n-1$ in $t$. In the following we write $\nu_\star$  for $(\mu_\star,\alpha_\star)$ for shortness.
	Since $l$ is a $L_\infty$-structure, it satisfies $l(\E^{\nu_\star} l(\E^{\nu_\star})) =0$ and since $\nu_\star$ is a Maurer-Cartan  element  up to order $n-1$ it satisfies $l(\E^{\nu_\star}) =0 \mod t^n$. So the $n^{th}$ order of $(\E^{\nu_\star} l(\E^{\nu_\star}))$ is given by $l(\E^{\nu_0}(\del\nu_n- R_n )) = \del\del \nu_n - \del R_n = \del R_n $ and has to vanish as claimed.
\end{proof}

\begin{corollary}
	If $\aHH^3(A) = \{0\}$, every deformation up to order $n$ can be extended to a full deformation. So especially for every 2-cocycle there exists a deformation, with it as first order term.
\end{corollary}
\begin{proof}
	We use the notation of the previous theorem. If $\mu_{(n)} = \sum_{i=0}^n \mu_i t^i, \alpha_{(n)}  \sum_{i=0}^n \alpha_i t^i$ is a deformation up to order $n$, the deformation equation $\del(\mu_{n+1},\alpha_{n+1}) = R_{n+1}$ can be solved, since $\aHH^3(A) = \{0\}$  and $(\mu_{(n)} + \mu_{n+1} t^{n+1}, \alpha_{(n)} + \alpha_{n+1} t^{n+1}$ is a deformation up to order $n+1$.  Continuing this for each order, gives a full deformation.  If $(\mu_1, \alpha_1)$ is a 2-cocycle $(\mu_0 + \mu_1 t, \alpha_0 +\alpha_1 t)$ is a deformation up to order $1$, so according to previous observation,  it can be extended to a deformation of $A$. 
\end{proof}

\begin{prop}
	If two deformations $(\mu_\star, \alpha_\star)$ and $(\mu'_\star,\alpha'_\star)$ have the same terms up to order $n-1$, then we have $\del(\mu_n - \mu'_n , \alpha_n - \alpha'_n) =0$ and there exists an equivalence up to order $n$ if there exists a linear map $S_n: A \to A$ such that $\del(S_n,0)= (\mu_n - \mu'_n , \alpha_n - \alpha'_n)$.
\end{prop}
\begin{proof}
 Let $S = \id + S_{n} t^n$, we have 
	\begin{align*}
		S(a {\star'} b)  &= \sum_{i=0}^n \mu'_i(a,b) t^i + S_n( ab) t^n  \mod t^{n+1} \text{   and }\\
		S(a) \star S(b) &= \sum_{i=0}^n \mu'_i(a,b) t^i + (S_n(a) b +  a S_n(b)) t^n  \mod t^{n+1}.
	\end{align*}
	So $S$ is an equivalence up to order $n$ if 
	\begin{align*}
		S_n(ab) + \mu'_n(ab) =S_n(a) b +  a S_n(b) + \mu_n(ab),
	\end{align*}
	which can be written as $\delmm S_n = \mu'_n - \mu_n$.  
	Similarly from $\alpha S - S \alpha' $  in order $n$, we get 
	\begin{align*}
		 \alpha_n  + \alpha_0 S_n - S_n -\alpha'_n 
		= \delma S_n   + \alpha_n - \alpha'_n =0.                                
	\end{align*}
	So $S =1+ S_n t^n$ is an equivalence up to order $n$.
										
	Using the $L_\infty$-structure, we calculate with $\nu_\star = (\mu_\star, \alpha_\star)$
	\begin{align*}
		l( (\E^{-{\nu_\star}}-1) (\E^{{\nu_\star} - {\nu'_\star}} -1) ) & = l(\E^{-{\nu_\star} } \E^{{\nu_\star} - {\nu'_\star}}) - l(\E^{-{\nu_\star}}) + l(\E^{{\nu_\star} - \tilde {\nu_\star}}) + l(1) \\
		                                               & = l(\E^{\nu_\star}) - l(\E^{{\nu_\star} - {\nu'_\star}})  = -l(\E^{{\nu_\star} - {\nu'_\star}}) .                   
	\end{align*}
	Since $\nu$ and $\nu'$ equal up to order $n-1$ the last term vanishes up to order $2(n-1)$, so especially in order $n$. The first  term in order $n$ equals  $l(\E^{\nu_0} (\nu_n - \nu'_n) )= \del(\nu_n - \mu'_n) $.
\end{proof}

\begin{corollary}
	Every deformation is equivalent to one of the form $\mu_\star = \mu_0 + \sum_{i=k}^\infty \mu_i t^i$ and $\alpha_\star = \sum_{i=k}^\infty \alpha_i t^i$, such that $(\mu_k,\alpha_k)$ is a 2-cocycle, which is not a coboundary. 
	
	If $\aHH^2(A) = \{0\}$ every deformation is equivalent to the undeformed algebra.
\end{corollary}
\begin{proof}
 If  $\mu_\star = \mu_0 + \sum_{i=k}^\infty \mu_i t^i$ and $\alpha_\star = \sum_{i=k}^\infty \alpha_i t^i$ is a deformation of $A$, then $(\mu_k,\alpha_k)$ is a 2-cocycle. If it is a  2-coboundary, by the previous proposition, there exists a equivalent deformation of the form  $(\mu_\star = \mu_0 + \sum_{i=k+1}^\infty \mu_i t^i, \alpha_\star = \alpha_0+\sum_{i=k+1}^\infty \alpha_i t^i)$, where the first non-trivial term is in order $k+1$. Repeating this if necessary, we arrive at a deformation such that the first nontrivial term is not a 2-coboundary. If $\aHH^2(A) = \{0\}$ every 2-cocylce is coboundary and the second statement is clear.
\end{proof}

Now, we define Hom-Poisson algebras and study their relationships to deformations of commutative Hom-associative algebras. 
	
\begin{defn}
	A Hom-Poisson algebra is a quatruple $(P,\mu,\p{\cdot,\cdot},\alpha)$, such that   $(P,\mu,\alpha)$ is a commutative Hom-associative algebra and $\p{\cdot,\cdot} : A \times A \to A$ is a skewsymmetric bilinear map, called Hom-Poisson bracket, such that 
	\begin{align}
		\p{\alpha( a),\p{b,c}} & = \p{\p{a,b},\alpha(c)} +\p{\alpha(b),\p{a,c}} &   & \text{(Hom-Jacobi identity),} \label{eq:homjacobi} \\
		\p{\alpha(a),bc}       & = \alpha(b)\p{a,c} + \p{a,b} \alpha(c)         &   & \text{(Hom-Leibniz identity),} \label{HLeibIdentity}\\ 
		\alpha(\p{a,b})        & = \p{\alpha(a),\alpha(b)}                      &   & \text{(multiplicativity),}     
	\end{align}
	for all $a,b,c$ in $A$.
\end{defn}

So a Hom-Poisson algebra is a Hom-Lie  and a commutative Hom-associative algebra, such  that the two operations are compatible.

A derivation of a Hom-Poisson algebra is a linear map  $\phi$ which is a  derivation with respect to the multiplication $\mu$ and  also a derivation with respect to  the Hom-Poisson bracket, i.e. it satisfies $\phi(ab) = \phi(a) b+ a\phi(b)$ and $\phi(\p{a,b}) = \p{\phi(a), b} + \p{a, \phi(b)}$. \\ An $\alpha$-derivation of a Hom-Poisson algebra  is a linear map which an $\alpha$-derivation for both the multiplication and   the Hom-Poisson bracket, i.e. it satisfies 
$\phi(ab) = \phi(a) \alpha(b) + \alpha(a) \phi(b)$
and $\phi(\p{a,b}) = \p{\phi(a), \alpha(b)} + \p{\alpha(a), \phi(b)}$.\\

We define  the star commutator
\begin{equation}
	[a,b]_\star = a \star b -b \star a.
\end{equation}
It is easy to see that it satisfies the Hom-Jacobi identity \eqref{eq:homjacobi} and the Hom-Leibniz identity \eqref{HLeibIdentity}. The following propositions shows that if $A$ is commutative its first order terms give rise to  a Hom-Poisson algebra on $A$.

\begin{prop}\label{th:assdef}
	Let $A$ be a commutative Hom-associative algebra and $A_\star$ be a deformation  of $A$. Then $\p{a,b} = \frac{1}{2 t}( a \star b - b \star a )\mod t$ defines a Hom-Poisson algebra on $A$
\end{prop}
\begin{proof}
	The bracket is obviously skewsymmetric.
	We have $[a,b]_\star = ab -ba  + t( \mu_1(a,b) - \mu_1(b,a)) + \O(t^2) = 2 t \p{a,b}  + \O(t^2)$.
	In second order we have 
	\begin{align*}
		[[a,b]_\star , \alpha_\star(c)]_\star = 4\p{ \p{a,b} ,\alpha(c)}. 
	\end{align*}
	Taking the cycling sum over this gives the  Hom-Jacobi identity.
	For the Hom-Leibniz identity we consider the Hom-Leibniz identity for the star commutator in first order this is 
	\begin{align}
		[ab,\alpha_\star(c)]                        & = 2\p{ab,\alpha(c)}                       \\
		\alpha_\star(a)[b,c] + [a,c]\alpha_\star(b) & = 2\alpha(a)\p{b,c} + 2\p{a,c} \alpha(b), 
	\end{align}
	which gives the desired identity. 
	Next we want to show that the  Poisson bracket is multiplicative with respect to $\alpha_0$. This follows from 
	\begin{equation} \label{eq:ser1}
		\begin{split}
			\alpha_\star([a,b]) &=  t \alpha_0(2\p{a,b}) +  t^2 \alpha_1(2\p{a,b}) + \alpha_2(\mu_2(a,b) - \mu_2(b,a))+ \O(t^3) \\
			[\alpha_\star(a), \alpha_\star(b)] &= t 2\p{\alpha_0(a),\alpha_0(b)} + t^2 (2\p{\alpha_1(a),\alpha_0 (b)} + 2\p{\alpha_0(a),\alpha_1(b)}  \\ &+  \mu_2(\alpha_0(a), \alpha_0(b) )  -  \mu_2(\alpha_0(b), \alpha_0(a) ) +  \O(t^3)
		\end{split}
	\end{equation}if one compares the first order terms.

\end{proof}
Dually for the deformation of a cocommutative Hom-coassociative algebra, $\delta(x) = \frac{1}{ 2 t} (\Delta(x) - \Delta^\opp(x)) \mod t$ defines a Hom-Poisson coalgebra.
	
As in the associative case the Poisson bracket is invariant under  equivalence and also here we have 
\begin{prop}
	Given two equivalent deformations of a commutative Hom-associative algebra $A$, the  corresponding Hom-Poisson brackets  are the same.
\end{prop} 
\begin{proof}
	If $\star$ and $\star'$ are equivalent there exists a $T= \id+ \sum_{i=1}^\infty  T_i t^i : A\ph \to A\ph$, such that $a \star' b = T^{-1}(T(a) \star T(b))$. Using $T^{-1} = \id - T_1 t + \O(t^2)$, this gives 
	\begin{equation}
		ab + \mu'_1(a,b)t = ab + t \big( \mu_1(a,b) + T_1(a)b +aT_1(b) -T_1(ab) \big).
	\end{equation} 
	From which follows $\p{a,b}' = \frac{1}{2} (\mu'_1(a,b) - \mu'_1(b,a)) =  \frac{1}{2} (\mu(a,b) - \mu(b,a)) = \p{a,b}$, as claimed.
\end{proof}

However different then in the associative case the Poisson bracket is not cohomologous to $\mu_1$. 
Instead we have that $(\p{\cdot,\cdot},0)$ is a cocycle. One can construct another cocycle, which is given by $(\frac{1}{2}(\mu_1(a,b) + \mu_1(b,a)) , \alpha_1)$. This however is not invariant under equivalence.  We suspect however that if $\alpha$ is invertible, one can construct a conjugate  $\alpha_0$-derivation out of it, which is invariant, as in the case where $\alpha_0 = \id$, see \Cref{th:defass}.\\

Now, we aim to study how the Poisson bracket and $\alpha_1$ change under Yau twists. 
	
\begin{prop}
	Let $(A,\mu,\alpha)$ be a Hom-associative algebra, $(A\ph,\star,\alpha_\star)$ a deformation of it  and $\phi = \sum_{i=0}^\infty t^i \phi_i$ be a morphism of $\star$, then the algebra $(A,\star_\phi,\phi \alpha_\star)$ obtained by Yau-twist is a deformation of $(A,\phi_0 \mu, \phi_0\alpha_0)$. In the case that $A$ is commutative the Hom-Poisson bracket is given by $\phi_0 \p{\cdot,\cdot}$, where $\p{\cdot,\cdot}$ denotes the Poisson bracket corresponding to $\star$. 
\end{prop}
\begin{proof}
	Using $a \star_\phi b = \phi ( a \star b) = \phi_0(ab) + t( \phi_0 \mu_1(a,b) + \phi_1 (ab) )+ \O( t^2)$ and $ \phi \alpha_\star = \phi_0 \alpha_0 + t( \phi_1 \alpha_0 +  \phi_0 \alpha_1) + \O(t)$ this is easy to check. We also note that $(\alpha_\phi)_1 =  \phi_1 \alpha_0 +  \phi_0 \alpha_1$.
\end{proof}
	
In the case $\phi_0= \id$, in the above proposition, we get a deformation of the same algebra, which also has the same Poisson bracket, but in general with different $\alpha_1$ so the deformations are not equivalent in general.

Next we want to consider the case, when $A$ is indeed associative, but is deformed into a  Hom-associative  algebra.  
	
\begin{prop}\label{th:defass}
	Let $(A,\mu,\id)$ be a commutative Hom-associative algebra and $(A,\star,\alpha)$ be a deformation of it, then $\alpha_1$ is a derivation of the associated Poisson algebra and invariant under equivalence.
\end{prop}
\begin{proof}
	Using the proof of \Cref{th:assdef}, especially the second order in \cref{eq:ser1}, we get that it is a derivation of the Poisson bracket, since the terms involving $\mu_2$ cancel. 
	Since  two equivalent transformations $\alpha_1$ differs by the commutator $[S_1, \alpha_0]$, the second statement is clear.
\end{proof}
	
So in this case we are closer to the associative case, since we get from the first order terms a canonical invariant algebraic structure, which at least in simple cases should be enough to determine all deformations up to equivalence.

In the associative case especially the deformation of the symmetric algebra $S(V)$ for a vector space $V$ is well studied and understood. In this case there is the well known theorem that any Poisson bracket admits a deformation, which was first proved by Kontsevich in \citep{kontsevich}. 
We can extend this to Hom-context in the following way.
\begin{theorem}
	Let $S(V)$ be the symmetric algebra over a vector space $V$, considered as a Hom-associative algebra, and $\p{\cdot,\cdot}$ a Poisson bracket on $S(V)$ and $\alpha_1$ a derivation of it, then there exists a deformation of it. 
\end{theorem}
\begin{proof}
	Using the theorem of Kontsevich there exits a deformation of the Poisson bracket, which we denote by $\tilde\star$. Since $\alpha_1$ is a derivation and Poisson derivation there exists a derivation $D = \alpha_1 + \sum_i D_i t^i$ of $\star$ as shown e.g. in \cite{Sharygin2016}.  With this $\E^{t D}$ is  an automorphism of the form $\phi = \id + t \alpha_1 + \O(t^2)$. Now the Yau twist $\star = \tilde{\star}_\phi$  has the desired properties. 
\end{proof}

\section{Examples}\label{sc:examples}

In the case, where $\alpha$ is invertible, the cohomology can be computed using \Cref{th:cohoyau}.
Let $V$ be a vector space with an endomorphism $\alpha$. Then we consider the symmetric algebra $S(V)$ over $V$ and $\tilde{\alpha}$ the extension of $\alpha$ to $S(V)$. With this we can define $S(V,\alpha)$ as the Yau twist of $S(V)$ by $\tilde{\alpha}$. If $\alpha$ is invertible this is isomorphic to the free commutative Hom-associative algebra over $(V,\alpha)$. 
	
Using the Hochschild-Kostant-Rosenberg Theorem, it is not difficult to see that  $$H^k(C(\tilde{\alpha }),\del_\mu) = \Lambda^k \Der(S(V)) \oplus \tilde{\alpha} \Lambda^{k-1} \Der(S(V)).$$ 
The differential $\del_\gamma$ is well defined on this complex. 
So we get $H(\gamma) =  \Lambda^k \Der_{\tilde{\alpha}}(S(V)) \oplus \factor{\tilde{\alpha} \Lambda^{k-1} \Der(S(V))}{\im \del_\gamma}$, where $\Der_\gamma(A) = \{ \phi \in \Der(A) | \phi \gamma = \gamma \phi\}$. 
If $V$ is finite dimensional and $\alpha$, and consequently $\tilde{\alpha}$, are diagonal  the second summand simplifies to $\tilde{\alpha} \Lambda^{k-1} \Der_{\tilde{\alpha}}(S(V))$.

Next we consider the case $\alpha =0$. In this case every bilinear map defines a Hom-associative algebra. It turns out that all differentials except low degrees are zero. In degree one, we have
$\delmm \phi (x,y) = x \phi(y) - \phi(xy) + \phi(x)y$ and in degree two $\delaa \psi(x,y) = \psi(xy)$ and $\delam \psi(x,y,z) = (xy)\psi(z) -\psi(x)(yz)$. So the cohomology in degree 1 consists of the derivations as always and starting from degree 4 it is the whole complex.
	
The other extreme case is when the multiplication is zero. In this case all differentials except of $\delma$ are zero. 
So the cohomology can be easily computed to be $$\aHH^n = \Hom_\alpha(A^{\otimes n},A) \oplus  \factor{\Hom(A^{\otimes (n-1)},A)}{\im \delma},$$ where $\Hom_\alpha(A^{\otimes n},A) = \{ \phi \in \Hom(A^{\otimes n},A) | \delma \phi =0\}$ are the linear maps commuting with $\alpha$. 
If $\alpha$ is diagonalizable, so is $\delma$ and we have $ \factor{\Hom(A^{\otimes (n-1)},A)}{\im \delma} \cong \Hom_\alpha(A^{\otimes (n-1)},A)$.

Next we want to give two examples, where we compute the cohomology -- or at least its dimension -- in low degrees explicitly using computer algebra systems. The first is a Hom-associative algebra, which is not of associative type the second one a truncated polynomial algebra.
	 
\paragraph{Hom-associative algebra not of associative type}
	
We consider the Hom-associative algebra, spanned by $e_1,e_2$ with structure map $\alpha(e_1) = e_1+ e_2,  \alpha(e_2) =0$ and multiplication $e_1 e_1= e_1 , e_i e_j =e_2$ for all other $i,j$.
Then using computer algebras system, we computed the lower cohomologies. It turns out that $\aHH(1) = \aHH^2 = 0, \dim(\aHH^3) =2$ and $\dim(\aHH^4) = 10$. The third cohomology is spanned by 
\begin{align*}
	\psi(e_1, e_2)     & = e_2,       &   
	\psi(e_2, e_1)     & = e_2,       &   
	\psi(e_2, e_2) &= e_2,           \\
	\phi(e_1, e_1,e_1) & = e_2,       &   
	\phi(e_1, e_1,e_2) & = 2 e_2-e_1, &   
	\phi(e_2, e_1,e_1) & = e_1,       &   
	\phi(e_2, e_1,e_2) & = e_2        &   
	\\ \text{and }
	\phi(e_1, e_1,e_1) & = -e_2,      &   
	\phi(e_1, e_2,e_1) &= -e_2,      &
	\phi(e_2, e_1,e_2) & = e_2,       &   
	\phi(e_2, e_2,e_2) &= e_2.       
\end{align*}

\paragraph{Truncated polynomial algebra}

We consider a vector space $V$ spanned by $x,y$ and an endomorphism $\alpha$ of it.
Then we can form the polynomial algebra $S(V)$ of it. We consider $A =\factor{S(V)}{S^3(V)}$. This is spanned by $\{1,x,y,x^2,xy,y^2\}$.  The map $\alpha$ can be extended to this as an algebra morphism, which we  denote by $\tilde{\alpha}$, so we can consider the Yau twist $A_\alpha$. 
The cohomology of course depends on $\alpha$. So we consider the different possibilities. 
We start with the case $\alpha = \id$. In  this case  $\dim(\aHH^1) =10,\dim(\aHH^2) =25$ and $\dim(\aHH^3) =41$.
	
In the case $\alpha= \lambda \id, \lambda \neq 1$, we have $\dim(\aHH^1) =4$, and the derivations are determined by
\begin{align*}
	\phi(x) & = \lambda_1 x + \lambda_2 y,  & 
	\phi(y) & = \lambda_3 x + \lambda_4 y. 
\end{align*} 
For the second cohomology, we have $\dim(\aHH^1) =7$. Four generators are of the form $(\phi_\psi, \psi)$ with $\psi = \alpha \phi$ for $\phi \in \aHH^1$. The rest is given by 
\begin{align*}
	\phi(x,y) & = \lambda_3 y^2 -\lambda_1 xy, & 
	\phi(y,x) & = \lambda_2 x^2 + \lambda_1 xy. 
\end{align*}
For the third cohomology, we get $\dim(\aHH^3) = 3$, so all comes from $\aHH^2$. 
	
In the case $\alpha(x) = \lambda(x), \alpha(y) = x + \lambda y$, this  $\alpha$ is not diagonalizable, we have $\dim(\aHH^1) = 2$,  and it is given by
\begin{align*}
	\phi(x) & = \lambda_1 x, &   
	\phi(y) & = \lambda_1 y +\lambda_2 x. 
\end{align*}
For the second cohomology, we have  $\dim(\aHH^2) = 3$, with two generators  coming from derivations and 
\begin{equation*}
	\phi(x,y) = x^2 = - \phi(y,x).
\end{equation*}
For the third cohomology, we get $\dim(\aHH^3) = 1$.  
	
The next case is a $\alpha(x) = \lambda x, \alpha(y) = \mu y$, with $\lambda^i \neq \mu^k$ for all $i,k \in \N$. In this 
case  $\dim(\aHH^1) = 2$,  namely
\begin{align*}
	\phi(x) & = \lambda_1 x, &   
	\phi(y) & = \lambda_2 y. 
\end{align*}
For the second cohomology, we get  $\dim(\aHH^2) = 3$, with two coming from derivations and
\begin{align*}
	\phi(x,y) = \lambda_1 xy. 
\end{align*}
For the third cohomology, we get $\dim(\aHH^3) =1 $.

As last case we want to consider a diagonal $\alpha$, but such that the eigenvalues are algebraically dependent, so we consider $\alpha(x) = 2x, \alpha(y) = 4y$.
In this case $\dim(\aHH^1) = 3$, namely the derivations determined by
\begin{align*}
	\phi(x) & = \lambda_1 x, & 
	\phi(y) & = \lambda_2 y + \lambda_3 x^2. 
\end{align*}
For the second cohomology, we get  $\dim(\aHH^2) = 6$, with three coming from derivations, one is the same as in the previous case  and
\begin{align*}
	&\phi(x,y) = \lambda_1 xy,                  \\
	&\phi(x,xy) = \lambda_2 y^2 = \phi(xy,x),   \\
	&\phi(y,x^2) = \lambda_2 y^2 = \phi(x^2,y), \\
	&\phi(x,x^2) = \phi(x^2,x) = \lambda_3 xy. 
\end{align*}

\section{$\alpha$-type Gerstenhaber-Schack cohomology for Hom-bialgebras}\label{sc:Hombialgeras}
	
In this section we  extend the $\alpha$-type Hochschild cohomology defined for Hom-associative algebras to Hom-bialgebras. We only give the definition here and reserve further study to a forthcoming paper.

As a shorthand notation, we set $H^{i,j} = \Hom(A^{\otimes i},A^{\otimes j})$, for a Hom-bialgebra $A$ and $i,j \in \N$. For $i$ or $j\leq 0$ we set $H^{i,j} = 0$. (This is necessary since the differentials we are going to define, would involve $\alpha^{-1}$ or $\beta^{-1}$ otherwise.)
We define a bicomplex $C_{GS}^{\bullet,\bullet}$ with 
\begin{equation}
	{C_{GS}}^{n,m} ={C_{GS}}^{n,m}_{\mu\Delta}  \oplus {C_{GS}}^{n,m}_{\alpha\Delta}  \oplus {C_{GS}}^{n,m}_{\mu\beta}\oplus {C_{GS}}^{n,m}_{\alpha\beta} 
	=   H^{n,m} \oplus  H^{n-1,m}  \oplus H^{n,m-1} \oplus H^{n-1,m-1} .
\end{equation}
	
	
We denote an element in ${C_{GS}}^{n,m}$ by $(\phi, \psi,\chi, \xi)$. The algebra differential $\del: {C_{GS}}^{n,m} \to {C_{GS}}^{n,m+1}$ is given by 
\begin{align}
	\del (\phi, \psi,\chi, \xi) = ( \delmm \phi - \delma \psi, \delma \phi - \delaa \psi, \delmm \chi - \delam \xi, \delma \chi - \delaa \xi) ,
\end{align}
where the differentials are the $\alpha$-type Hochschild differentials defined before, where the left action on $A^{\otimes m}$ in ${C_{GS}}^{n,m}_{\alpha\Delta}$ and ${C_{GS}}^{n,m}_{\mu\Delta}$ is the usual one and  the left-action on $A^{\otimes (m-1)}$ in ${C_{GS}}^{n,m}_{\alpha\beta}$ and ${C_{GS}}^{n,m}_{\mu\beta}$ is given by $\beta(x) \cdot y$, where $\cdot$ denotes the usual left-action on $A^{\otimes (m-1)}$, and similarly for the right-action.

Dually the coaction on $A^{\otimes (m-1)}$  in  ${C_{GS}}^{n,m}_{\alpha\beta}$ and ${C_{GS}}^{n,m}_{\alpha\Delta}$ is given by $\alpha(x_{(-1)}) \otimes x_{(0)}$, where $x \mapsto x_{(-1)} \otimes x_{(0)}$ denotes the usual left coaction on $A^{\otimes (m-1)}$.

Since there are a lot of different maps involved we give a diagram depicting all maps, which start in $C_{GS}^{i,j}$ :
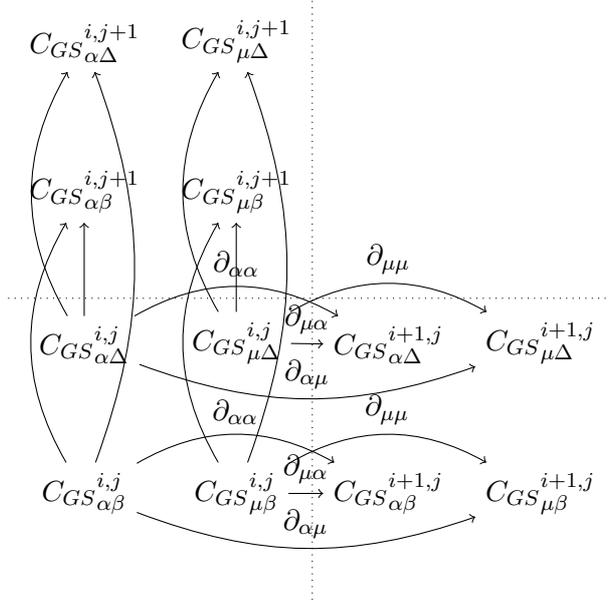
\begin{center}
	\begin{tikzpicture}[scale=2.0,above]
		\node (ab11) at (1,1) {${C_{GS}}^{i,j}_{\alpha\beta}$};
		\node (mb11) at (2,1) {${C_{GS}}^{i,j}_{\mu\beta}$};
		\node (ad11) at (1,2) {${C_{GS}}^{i,j}_{\alpha\Delta}$};
		\node (md11) at (2,2) {${C_{GS}}^{i,j}_{\mu\Delta}$};
		\node (ab21) at (3,1) {${C_{GS}}^{i+1,j}_{\alpha\beta}$};
		\node (mb21) at (4,1) {${C_{GS}}^{i+1,j}_{\mu\beta}$};
		\node (ad21) at (3,2) {${C_{GS}}^{i+1,j}_{\alpha\Delta}$};
		\node (md21) at (4,2) {${C_{GS}}^{i+1,j}_{\mu\Delta}$};
		\node (ab12) at (1,3) {${C_{GS}}^{i,j+1}_{\alpha\beta}$};
		\node (mb12) at (2,3) {${C_{GS}}^{i,j+1}_{\mu\beta}$};
		\node (ad12) at (1,4) {${C_{GS}}^{i,j+1}_{\alpha\Delta}$};
		\node (md12) at (2,4) {${C_{GS}}^{i,j+1}_{\mu\Delta}$};
		\draw[dotted] (2.5,0.5) -- (2.5,4.5);
		\draw[dotted] (0.5,2.5) -- (4.5,2.5);
		\path[->] (ab11) edge[bend left]  node {$\delaa$} (ab21); 
		\path[->] (ab11) edge[bend right=20]  node {$\delam$} (mb21); 
		\path[->] (mb11) edge node {$\delma$} (ab21); 
		\path[->] (mb11) edge[bend left]  node {$\delmm$} (mb21); 
		\path[->] (ad11) edge[bend left]  node {$\delaa$} (ad21); 
		\path[->] (ad11) edge[bend right=20]  node {$\delam$} (md21); 
		\path[->] (md11) edge  node {$\delma$} (ad21); 
		\path[->] (md11) edge[bend left]  node {$\delmm$} (md21); 
																	
		\path[->] (ab11) edge[bend left]  node {} (ab12); 
		\path[->] (ab11) edge[bend right=20]  node {} (ad12); 
		\path[->] (ad11) edge node {} (ab12); 
		\path[->] (ad11) edge[bend left]  node {} (ad12); 
		\path[->] (mb11) edge[bend left]  node {} (mb12); 
		\path[->] (mb11) edge[bend right=20]  node {} (md12); 
		\path[->] (md11) edge  node {} (mb12); 
		\path[->] (md11) edge[bend left]  node {} (md12); 	
	\end{tikzpicture}
\end{center}
	
\begin{prop}
	The differentials  defined above make $C_{GS}$ to a bicomplex.
\end{prop}
\begin{proof}
	From the  known differential of Hom-bialgebras \citep{MR3640817} its follows that $\delmm \del_{\Delta\Delta} = \del_{\Delta\Delta} \delmm$. By doing  essentially  the same calculation, one gets the same for all parts ${C_{GS}}^{n,m}_{ij} \to {C_{GS}}^{n+1,m+1}_{ij}$ for $i \in \{\mu,\alpha\}$ and $j\in\{\Delta,\beta\}$.
												
	Next we consider the part ${C_{GS}}^{n,m}_{\mu\Delta} \to {C_{GS}}^{n,m}_{\alpha\beta}$, for this we compute
	\begin{align*}
		\del_{\Delta\beta} \delma  \phi(x_1,\dots,x_{n}) = & \beta^{\otimes m} \alpha^{\otimes m} \phi(x_1,\dots,x_{n})                                                
		- \beta^{\otimes m} \phi(\alpha(x_1),\dots, \alpha(x_n))  \\
		                                                   & - \alpha^{\otimes m} \phi( \beta(x_1,\dots, \beta(x_n)) + \phi(\beta(\alpha(x_1),\dots,\beta(\alpha(x_n)) \\
		=                                                  & \delma \del_{\Delta\beta} \phi(x_1,\dots,x_{n}) .                                                          
	\end{align*}
	This holds, since $\alpha$ and $\beta$ commute.

	${C_{GS}}^{n,m}_{\mu\Delta} \to {C_{GS}}^{n,m}_{\alpha\beta}$
	\begin{small}
		\begin{align*}
			\del_{\Delta\Delta} & \delma \phi(x_1,\dots,x_n)  =  \alpha \mu^n_\alpha(x^{(1)}_1,\dots,x^{(1)}_n) \otimes (\delma \phi)   (x_1^{(2)},\dots,x_n^{(2)}) \\
			                    & + \sum_{i=1}^m (-1)^i \Delta_i (\delma \phi)(x_1,\dots,x_n)                                                                       
			+ (-1)^{m+1}   (\delma \phi)   (x_1^{(1)},\dots,x_n^{(1)}) \otimes \alpha \mu^n_\alpha(x^{(2)}_1,\dots,x^{(2)}_n)          \\
			                    & = \alpha \mu^n_\alpha(x^{(1)}_1,\dots,x^{(1)}_n) \otimes (\alpha^{\otimes m} \phi)   (x_1^{(2)},\dots,x_n^{(2)})                  
			- \alpha \mu^n_\alpha(x^{(1)}_1,\dots,x^{(1)}_n) \otimes  \phi   (\alpha(x_1^{(2)}),\dots,\alpha(x_n^{(2)}))               \\
			                    & + \sum_{i=1}^m (-1)^i \Delta_i (\alpha^{\otimes m} \phi)(x_1,\dots,x_n)                                                           
			- \sum_{i=1}^m (-1)^i \Delta_i  \phi(\alpha(x_1),\dots,\alpha(x_n))                                                        \\
			                    & + (-1)^{m+1}   (\alpha^{\otimes m} \phi)(x_1^{(1)},\dots,x_n^{(1)}) \otimes \alpha \mu^n_\alpha(x^{(2)}_1,\dots,x^{(2)}_n)        \\
			                    & - (-1)^{m+1}   \phi(\alpha(x_1^{(1)}),\dots,\alpha(x_n^{(1)})) \otimes  \alpha\mu^n_\alpha(x^{(2)}_1,\dots,x^{(2)}_n)             
		\end{align*}
		\begin{align*}
			\delma \del_{\Delta\Delta} \phi(x_1,\dots,x_n) & =  \alpha^{\otimes m+1}( \del_{\Delta\Delta} \phi)(x_1,\dots,x_n) - ( \del_{\Delta\Delta} \phi)(x_1,\dots,x_n)               \\
			                                               & = \alpha^{\otimes m+1} (\mu^n_\alpha(x_1^{(1)},\dots,x_n^{(1)}) \otimes \phi(x_1^{(2)}, \dots, x_n^{(2)})                    \\
			                                               & +  \sum_{i=1}^m (-1)^i   \alpha^{\otimes m+1} \Delta_i \phi(x_1,\dots,x_n)                                                   \\ 
			                                               & + (-1)^{m+1} \alpha^{m+1} (\phi(x_1^{(1)},\dots,x_n^{(1)}) \otimes \mu^n_\alpha(x_1^{(2)}, \dots, x_n^{(2)}))                \\
			                                               & - \mu^n_\alpha(\alpha(x_1)^{(1)},\dots,\alpha(x_n)^{(1)}) \otimes \phi(\alpha(x_1)^{(2)},\dots,\alpha(x_n)^{(2)})            \\
			                                               & - \sum_{i=1}^m (-1)^i \Delta_i  \phi(\alpha(x_1),\dots,\alpha(x_n))                                                          \\
			                                               & -(-1)^{m+1}  \phi(\alpha(x_1)^{(1)},\dots,\alpha(x_n)^{(1)}) \otimes \mu^n_\alpha(\alpha(x_1)^{(2)},\dots,\alpha(x_n)^{(2)}) 
		\end{align*}
	\end{small}
	Using that $\alpha$ is a morphism of $\Delta$, it is easy to see that the two sides agree.

	${C_{GS}}^{n,m}_{\alpha\beta} \to {C_{GS}}^{n,m}_{\mu\Delta}$
	\begin{align*}
		(\del_{\beta\Delta} & \delmm \xi) (x_1,\dots,x_{n+1})=                                                                                                                                        
		\Delta\beta^{m-2} \mu^n_\alpha(w_1^{(1)},\dots,x_n^{(1)}) \otimes (\delam \xi)(x_1^{(2)},\dots,x_1^{(2)}) \\
		                    & -  (\delam \xi)(x_1^{(1)},\dots,x_1^{(1)}) \otimes \Delta\beta^{m-2} \mu^n_\alpha(w_1^{(2)},\dots,x_n^{(2)})                                                            \\
		                    & = \Delta \beta^{m-2} \mu^n_\alpha(x_1^{(1)},\dots,x_n^{(1)})  \otimes  \beta \alpha^{n-2}(x_1^{(2)}x_2^{(2)}) \cdot \xi(x_3^{(2)},\dots,x_{n+1}^{(2)})                  \\
		                    & - (-1)^n  \Delta \beta^{m-2} \mu^n_\alpha(w_1^{(1)},\dots,x_n^{(1)})  \otimes \xi(x_1^{(2)},\dots,x_{n-1}^{(2)}) \cdot  \beta \alpha^{n-2}(x_{n}^{(2)}x_{n+1}^{(2)})    \\
		                    & -    \beta \alpha^{n-2}(x_1^{(1)}x_2^{(1)}) \cdot \xi(x_3^{(1)},\dots,x_{n+1}^{(1)}) \otimes  \Delta \beta^{m-2} \mu^n_\alpha(x_1^{(2)},\dots,x_n^{(2)})                \\
		                    & + (-1)^{n}   \xi(x_1^{(1)},\dots,x_{n-1}^{(1)}) \cdot  \beta \alpha^{n-2}(x_{n}^{(1)}x_{n+1}^{(1)}) \otimes  \Delta \beta^{m-2} \mu^n_\alpha(x_1^{(2)},\dots,x_n^{(2)}) 
	\end{align*}
	On the other side we get 
	\begin{align*}
		(\delam \del_{\beta\Delta} \xi) & (x_1,\dots,x_{n+1}) =  \alpha^{n-2}(x_1x_2) \cdot (\del_{\beta\Delta} \xi)(x_3,\dots,x_{n+1})                                                                            \\
		                                & - (\del_{\beta\Delta} \xi)(x_1,\dots,x_{n-1}) \cdot  \alpha^{n-2}(x_nx_{n+1})                                                                                            \\
		                                & = \alpha^{n-2}(x_1x_2) \cdot \left( \Delta \alpha \beta^{m-2} \mu^{n-i}_\alpha(x_3^{(1)},\dots,x_{n+1}^{(1)}) \otimes \xi(x_3^{(2)},\dots,x_{n+1}^{(2)}) \right)         \\
		                                & -   \alpha^{n-2}(x_1x_2) \cdot \left( \xi(x_3^{(1)},\dots,x_{n+1}^{(1)}) \otimes \Delta \alpha \beta^{m-2} \mu^{n-i}_\alpha(x_3^{(2)},\dots,x_{n+1}^{(2)}) \right)       \\
		                                & -   \left( \Delta \alpha \beta^{m-2} \mu^{n-i}_\alpha(x_1^{(1)},\dots,x_{n-1}^{(1)}) \otimes \xi(x_1^{(2)},\dots,x_{n-1}^{(2)}) \right) \cdot \alpha^{n-2}(x_nx_{n+1})   \\
		                                & +   \ \left( \xi(x_1^{(1)},\dots,x_{n-1}^{(1)}) \otimes \Delta \alpha \beta^{m-2} \mu^{n-i}_\alpha(x_1^{(2)},\dots,x_{n-1}^{(2)}) \right) \cdot \alpha^{n-2}(x_nx_{n+1}) 
	\end{align*}
	Using \eqref{eq:mod1}  and \eqref{eq:mu1} we get 
	\begin{small}
		\begin{align*}
			  & \alpha^{n-2}(x_1x_2) \cdot \left( \Delta \alpha \beta^{m-2} \mu^{n-i}_\alpha(x_3^{(1)},\dots,x_{n+1}^{(1)}) \otimes \xi(x_3^{(2)},\dots,x_{n+1}^{(2)}) \right)                                                                                \\
			  & =\left(\beta^{m-2}\alpha^{n-2}(x_1^{(1)}x_2^{(1)}) \cdot \Delta \alpha \beta^{m-2} \mu^{n-i}_\alpha(x_3^{(1)},\dots,x_{n+1}^{(1)})\right) \otimes \left(\beta\alpha^{n-2}(x_1^{(2)}x_2^{(2)})\cdot \xi(x_3^{(2)},\dots,x_{n+1}^{(2)}) \right) \\
			  & =\left( \Delta a \beta^{m-2} \mu^{n+1}_\alpha(x_1^{(1)},\dots,x_{n+1}^{(1)})\right) \otimes \left(\beta\alpha^{n-2}(x_1^{(2)}x_2^{(2)})\cdot \xi(x_3^{(2)},\dots,x_{n+1}^{(2)}) \right)                                                       
		\end{align*}
	\end{small}
	Similarly the other three terms can be manipulated, which shows that the two sides agree.
												
	${C_{GS}}^{n,m}_{\alpha\Delta} \to {C_{GS}}^{n,m}_{\mu\Delta}$
	\begin{small}
		\begin{align*}
			\delmm & \del_{\beta\Delta} \chi(x_1,\dots,x_{n+1})  = \alpha^{n-1}(x_1) \cdot (\del_{\beta\Delta} \chi) (x_2,\dots,x_{n+1})                                                                                   \\
			       & +\sum_{i=1}^n (-1)^i \del_{\beta\Delta} \chi)(\alpha(x_1), \dots, x_i x_{i+1},\dots,\alpha(x_n))                                                                                                      
			+ (-1)^{n+1}  (\del_{\beta\Delta} \chi)(x_1, \dots, x_n) \cdot \alpha^{n-1}( x_{n+1})                                                                                                                 \\
			       & =\alpha^{n-1}(x_1) \cdot \left(\Delta \beta^{m-2} \mu^{n}_\alpha(x_2^{(1)},\dots,x_{n+1}^{(1)}) \otimes \chi(x_2^{(2)},\dots,x_{n+1}^{(2)})  \right)                                                  \\
			       & - \alpha^{n-1}(x_1) \cdot \left( \chi(x_2^{(1)},\dots,x_{n+1}^{(1)}) \otimes \Delta \beta^{m-2} \mu^{n}_\alpha(x_2^{(2)},\dots,x_{n+1}^{(2)}) \right)                                                 \\
			       & + \sum_{i=1}^n (\beta^{m-2} \mu^n_\alpha(\alpha(x_1)^{(1)},\dots,(x_ix_{i+1})^{(1)},\dots,\alpha(x_{n+1}^{(1)})) \otimes  \chi(\alpha(x_1)^{(2)},\dots,(x_ix_{i+1})^{(2)},\dots,\alpha(x_{n+1}^{(2)}) \\
			       & - \sum_{i=1}^n   \chi(\alpha(x_1)^{(1)},\dots,(x_ix_{i+1})^{(1)},\dots,\alpha(x_{n+1}^{(1)}) \otimes \beta^{m-2} \mu^n_\alpha(\alpha(x_1)^{(2)},\dots,(x_ix_{i+1})^{(2)},\dots,\alpha(x_{n+1}^{(2)})) \\
			       & + (-1)^{n+1} \left(\Delta \beta^{m-2} \mu^{n}_\alpha(x_1^{(1)},\dots,x_{n}^{(1)}) \otimes \chi(x_1^{(2)},\dots,x_n^{(2)})  \right) \cdot \alpha^{n-1}(x_{n+1})                                        \\
			       & -(-1)^{n+1} \ \left( \chi(x_1^{(1)},\dots,x_n^{(1)}) \otimes \Delta \beta^{m-2} \mu^{n}_\alpha(x_1^{(2)},\dots,x_{n}^{(2)}) \right) \cdot \alpha^{n-1}(x_{n+1})                                       
		\end{align*}
	\end{small}  																 
		
	We note that $\mu^n_\alpha(\alpha(x_1),\dots,(x_ix_{i+1}),\dots,\alpha(x_{n+1})) = \mu^{n+1}_\alpha(x_1,\dots,x_{n+1})$.
	\begin{align*}
		(\del_{\beta\Delta} \delmm \chi) & (x_1,\dots,x_{n+1})= \Delta \beta^{m-2} \mu^{n+1}_\alpha(x_1^{(1)},\dots,x_{n+1}^{(1)}) \otimes (\del_{\beta\Delta}\chi)(x_1^{(2)},\dots,x_{n+1}^{(2)})                          \\
		                                 & -     (\del_{\beta\Delta}\chi)(x_1^{1)},\dots,x_{n+1}^{(1)}) \otimes  \Delta \beta^{m-2} \mu^{n+1}_\alpha(x_1^{(2)},\dots,x_{n+1}^{(2)})                                         \\
		                                 & = \Delta \beta^{m-2} \mu^{n+1}_\alpha(x_1^{(1)},\dots,x_{n+1}^{(1)}) \otimes \beta \alpha^{n-1}(x_1^{(2)}) \cdot \chi(x_2^{(2)},\dots,x_{n+1}^{(2)})                             \\
		                                 & + \sum_{i=1}^n (-1)^i \Delta \beta^{m-2} \mu^{n+1}_\alpha(x_1^{(1)},\dots,x_{n+1}^{(1)}) \otimes \chi(\alpha(x_1^{(2)}),\dots,x_i^{(2)}x_{i+1}^{(2)},\dots,\alpha(x_{n+1}^{(2)}) \\
		                                 & + (-1)^{n+1} \Delta \beta^{m-2} \mu^{n+1}_\alpha(x_1^{(1)},\dots,x_{n+1}^{(1)}) \otimes \beta   \chi(x_1^{(2)},\dots,x_{n}^{(2)}) \cdot \alpha^{n-1}(x_{n+1}^{(2)})              \\
		                                 & -  \beta \alpha^{n-1}(x_1^{(1)}) \cdot \chi(x_2^{(1)},\dots,x_{n+1}^{(1)}) \otimes \Delta \beta^{m-2} \mu^{n+1}_\alpha(x_1^{(2)},\dots,x_{n+1}^{(2)})                            \\
		                                 & - \sum_{i=1}^n (-1)^i \chi(\alpha(x_1^{(1)}),\dots,x_i^{(2)}x_{i+1}^{(1)},\dots,\alpha(x_{n+1}^{(1)}) \otimes \Delta \beta^{m-2} \mu^{n+1}_\alpha(x_1^{(2)},\dots,x_{n+1}^{(2)}) \\
		                                 & - (-1)^{n+1}  \beta   \chi(x_1^{(1)},\dots,x_{n}^{(1)}) \cdot \alpha^{n-1}(x_{n+1}^{(1)}) \otimes \Delta \beta^{m-2} \mu^{n+1}_\alpha(x_1^{(2)},\dots,x_{n+1}^{(2)})             
	\end{align*}
	Again using  \eqref{eq:mod1}  and \eqref{eq:mu1} we get  
	\begin{align*}
		\alpha^{n-1}(x_1) \cdot \left(\Delta \beta^{m-2} \mu^{n}_\alpha(x_2^{(1)},\dots,x_{n+1}^{(1)}) \otimes \chi(x_2^{(2)},\dots,x_{n+1}^{(2)})  \right) = \\ 
		\Delta \beta^{m-2} \mu^{n+1}_\alpha(x_1^{(1)},\dots,x_{n+1}^{(1)}) \otimes \beta \alpha^{n-1}(x_1^{(2)}) \cdot \chi(x_2^{(2)},\dots,x_{n+1}^{(2)}).   
	\end{align*} With this and similar equalities it is easy to see that the two sides agrees.

\end{proof}

This describes the deformation of a general Hom-bialgebra. More about  deformations of Hom-bialgebras will be given in a forthcoming paper.

If one wants to consider only the case, where $\alpha = \beta$, i.e.  the structure maps are the same, and one wants the deformation to respect this or only consider algebras where this is the case  one has two identify $C^{i+1,j}_{\alpha\Delta}$ with $C^{i,j+1}_{\mu\beta}$ and set $C^{i,j}_{\alpha\beta}$ to $\{0\}$. Note that this does not respect the bicomplex structure, so the resulting complex is no longer a bicomplex.
For example we have $C_{GS}^2 = H^{2,1} \oplus H^{1,1} \oplus H^{1,2}$. 
Since this is often considered we want to give some more details.
	
\begin{prop}
	Let $(A,\mu,\alpha,\Delta,\alpha)$ be a Hom-bialgebra. Then the subspace of ${C_{GS}}$ given in degree $n$ by elements of the form $(\phi_{i,j},\psi_{i,j},\chi_{i,j},\xi_{i,j})_{i+j=n}$, with $\phi_{i,j} \in {C_{GS}}^{i,j}_{\mu\Delta}$, $\psi_{i,j} \in {C_{GS}}^{i,j}_{\alpha\Delta}$, $\chi_{i,j} \in {C_{GS}}^{i,j}_{\alpha\Delta}$ and $\xi_{i,j} \in {C_{GS}}^{i,j}_{\alpha\beta}$ for all $i,j $,  which satisfy $\psi_{i,j}= \chi_{i-1,j+1}$ and $\xi_{i,j}=0$,  is a subcomplex, i.e. it is preserved by the differential.
\end{prop}
\begin{proof} We need to verify that $\del (\phi_{i,j},\psi_{i,j},\chi_{i,j},\xi_{i,j})_{i+j=n}$ again satisfies the two constraints. 
	For the first we have 
	$\delmm \chi_{i,j+1} = \delaa  \psi_{i+1,j}$ and 
	$\delma \phi_{i,j} = \del_{\Delta\beta} \phi_{i,j} $.										
	For the second we have 
	$\delma \chi_{i-1,j} =  \del_{\Delta\beta} \psi_{i,j-1}  $ since $\alpha= \beta$.
\end{proof}
	
So the complex can be given by 
\begin{align}
	{C_{GS}}^i = & \bigoplus_{j= 1}^{i-1} {C_{GS}}_\mu^{j,i-j} \oplus \bigoplus_{j=1}^{i-2} {C_{GS}}_\alpha^{j,i-j-1} \\
	=            & \bigoplus_{j= 1}^{i-1} H^{j,i-j} \oplus \bigoplus_{j=1}^{i-2} H^{j,i-j-1} .                         
\end{align}
Note that the degree of ${C_{GS}}_\mu^{i,j}$ is $i+j$, while the degree of ${C_{GS}}_\alpha^{i,j}$ is $i+j+1$.
The differential consists of maps
\begin{align*}
	\delmm :      & {C_{GS}}_\mu^{i,j} \to {C_{GS}}_\mu^{i+1,j}:                                                            \\ 
	(\delmm \phi) & (x_1, \dots, x_{i+1})  = \alpha^{i-1}(x_1)\cdot\phi(x_2,\dots,x_{i+1})                                  \\
	              & + \sum_{k=1}^i  (-1)^k \phi(\alpha(x_1), \dots, x_k x_{k+1}, \dots, \alpha(x_{i+1}))                    
	+ (-1)^{i+1} \phi(x_1,\dots,x_i) \cdot \alpha^{i-1}(x_{i+1}),      \\
	\delma:       & {C_{GS}}_\mu^{i,j} \to {C_{GS}}_\alpha^{i,j}:                                                           \\
	(\delma \phi) & (x_1,\dots,x_i)        = \alpha^{\otimes j} \phi(x_1,\dots, x_i) - \phi(\alpha(x_1),\dots,\alpha(x_i)), \\
	\delaa :      & {C_{GS}}_\alpha^{i,j} \to {C_{GS}}_\alpha^{i+1,j}:                                                      \\ 
	(\delaa \psi) & (x_1, \dots,x_{i+1})   = \alpha^{i}(x_1) \cdot \psi(x_2,\dots,x_{i+1})                                  \\
	              & + \sum_{k=1}^i  (-1)^k \psi(\alpha(x_1), \dots, x_k x_{k+1}, \dots, \alpha( x_{i+1}))                   
	+ (-1)^{i+1} \psi(x_1,\dots,x_i) \cdot \alpha^{i}(x_{i+1}),                           \\
	\delam:       & {C_{GS}}_\alpha^{i,j} \to {C_{GS}}_\mu^{i+2,j}:                                                         \\ 
	(\delam\psi)  & (x_1, \dots, x_{i+2})    =                                                                              
	\alpha^{i}(x_1 x_2) \cdot \psi(x_3,\dots, x_{i+2}) - \psi(x_1, \dots,x_{i}) \cdot\alpha^i (x_{i+1} x_{i+2}),
\end{align*}
Dually we define the differentials
$\del_{\Delta\Delta}:{C_{GS}}_\mu^{i,j} \to {C_{GS}}_\mu^{i,j+1}$,
$\del_{\beta\beta}:{C_{GS}}_\alpha^{i,j} \to {C_{GS}}_\alpha^{i,j+1}$ and
$ \del_{\beta\Delta}:{C_{GS}}_\alpha^{i,j} \to {C_{GS}}_\mu^{i,j+2}$. 
	
The different parts of the differential can be seen in the following diagram:
\begin{center}
	\begin{tikzpicture}[scale=2.5]
		\node (m11) at (1,1) {${C_{GS}}_\mu^{i,j}$};
		\node (m12) at (1,2)  {${C_{GS}}_\mu^{i,j+1}$};
		\node (m21) at (2,1)   {${C_{GS}}_\mu^{i+1,j}$};
		\node (m31) at (3,1)   {${C_{GS}}_\mu^{i+2,j}$};
		\node (m13) at (1,3)   {${C_{GS}}_\mu^{i,j+2}$};
		\node (a11) at (1.5,1.5)  {${C_{GS}}_\alpha^{i,j}$};
		\node (a12) at (1.5,2.5)  {${C_{GS}}_\alpha^{i,j+1}$};
		\node (a21) at (2.5,1.5)   {${C_{GS}}_\alpha^{i+1,j}$};
		\draw[->] (m11) -- (m12);
		\draw[->] (m11) -- (m21);
		\draw[->] (m11) -- (a11);
		\draw[->] (a11) -- (m13);
		\draw[->] (a11) -- (m31);
		\draw[->] (a11) -- (a12);
		\draw[->] (a11) -- (a21);		
	\end{tikzpicture} 
\end{center}

\bibliographystyle{bibstyle}
\bibliography{paperpr}

\def\cprime{$'$}
\ifx\undefined\bysame
\newcommand{\bysame}{\leavevmode\hbox to3em{\hrulefill}\,}
\fi
\begin{thebibliography}{LGMT15}

\bibitem[AEM11]{homcoho}
F.~Ammar, Z.~Ejbehi, and A.~Makhlouf, {\em Cohomology and deformations of
  {H}om-algebras}, J. Lie Theory \textbf{ 21} (2011), no.~4, 813--836.

\bibitem[DM17]{MR3640817}
K.~Dekkar and A.~Makhlouf, {\em Gerstenhaber--{S}chack cohomology for
  {H}om-bialgebras and deformations}, Comm. Algebra \textbf{ 45} (2017),
  no.~10, 4400--4428.

\bibitem[FMY09]{fregier1}
Y.~Fr{\'e}gier, M.~Markl, and D.~Yau, {\em The {$L_\infty$}-deformation complex
  of diagrams of algebras}, New York J. Math. \textbf{ 15} (2009), 353--392.

\bibitem[FZ15]{fregier2}
Y.~Fr{\'e}gier and M.~Zambon, {\em Simultaneous deformations of algebras and
  morphisms via derived brackets}, J. Pure Appl. Algebra \textbf{ 219} (2015),
  no.~12, 5344--5362.

\bibitem[Ger64]{gerstenhaber1}
M.~Gerstenhaber, {\em On the deformation of rings and algebras}, Ann. of Math.
  (2) \textbf{ 79} (1964), 59--103.

\bibitem[GS83]{gerstenhabermor1}
M.~Gerstenhaber and S.~D. Schack, {\em On the deformation of algebra morphisms
  and diagrams}, Trans. Amer. Math. Soc. \textbf{ 279} (1983), no.~1, 1--50.

\bibitem[GS85]{gerstenhabermor2}
M.~Gerstenhaber and S.~D. Schack, {\em On the cohomology of an algebra
  morphism}, J. Algebra \textbf{ 95} (1985), no.~1, 245--262.

\bibitem[HLS06]{silvestrov06}
J.~T. Hartwig, D.~Larsson, and S.~D. Silvestrov, {\em Deformations of {L}ie
  algebras using {$\sigma$}-derivations}, J. Algebra \textbf{ 295} (2006),
  no.~2, 314--361.
  
\bibitem[Kon03]{kontsevich}
M.~Kontsevich, {\em Deformation quantization of {P}oisson manifolds}, Lett.
  Math. Phys. \textbf{ 66} (2003), no.~3, 157--216.

\bibitem[LGMT18]{UEHomLie}
C.~Laurent-Gengoux, A.~Makhlouf, and J.~Teles, {\em Universal algebra of a
  {H}om-{L}ie algebra and group-like elements}, J. Pure Appl. Algebra \textbf{
  222} (2018), no.~5, 1139--1163.


\bibitem[LV12]{MR2954392}
J.-L. Loday and B.~Vallette, {\em Algebraic operads}, Grundlehren der
  Mathematischen Wissenschaften [Fundamental Principles of Mathematical
  Sciences], vol. 346, Springer, Heidelberg, 2012.
  
  \bibitem[MS08]{homalg}
A.~Makhlouf and S.~D. Silvestrov, {\em Hom-algebra structures}, J. Gen. Lie
  Theory Appl. \textbf{ 2} (2008), no.~2, 51--64.

\bibitem[Mar10]{MR2812919}
M.~Markl, {\em Intrinsic brackets and the {$L_\infty$}-deformation theory of
  bialgebras}, J. Homotopy Relat. Struct. \textbf{ 5} (2010), no.~1, 177--212.

\bibitem[MS10]{makhlouf_ParamFormDefHomAssLie}
A.~Makhlouf and S.~Silvestrov, {\em Notes on 1-parameter formal deformations of
  {H}om-associative and {H}om-{L}ie algebras}, Forum Math. \textbf{ 22} (2010),
  no.~4, 715--739.
  
\bibitem[Sha17]{Sharygin2016}
G.~Sharygin, {\em Deformation quantization and the action of {P}oisson vector
  fields}, Lobachevskii J. Math. \textbf{ 38} (2017), no.~6, 1093--1107.
\end{thebibliography}

\end{document}